\documentclass[a4paper,11pt,hidelinks]{amsart}  

\usepackage{graphicx} 
\usepackage[a4paper, margin=2 cm]{geometry}
\usepackage{hyperref}
\usepackage{amsthm}
\usepackage{thm-restate}
\usepackage{geometry}
\usepackage{amssymb}
\usepackage{amsmath}
\usepackage{cleveref}
\usepackage{soul}
\usepackage{comment}
\usepackage{setspace}
\usepackage{xcolor}
\usepackage[shortlabels]{enumitem}
\usepackage{listings}
\usepackage{blkarray}
\usepackage{lmodern}
\usepackage[english]{babel}
\usepackage{float}
\usepackage{tikz}
\usetikzlibrary{arrows.meta,arrows}
\usetikzlibrary{arrows,decorations.markings}
\usetikzlibrary{decorations.pathmorphing}
\usepackage[square,sort,comma,numbers]{natbib}
\newtheorem{theorem}{Theorem}[section]
\newtheorem{lemma}[theorem]{Lemma}
\newtheorem{corollary}[theorem]{Corollary}

\newtheorem{observation}[theorem]{Observation}
\newtheorem{claim}[theorem]{Claim}
\newtheorem*{claim*}{Claim}
\newtheorem{proposition}[theorem]{Proposition}
\newtheorem{conjecture}[theorem]{Conjecture}
\newtheorem{example}[theorem]{Example}

\theoremstyle{definition}
\newtheorem{definition}[theorem]{Definition}
\newtheorem*{definition*}{Definition}

\newcommand{\id}{{0}}

 \newcommand{\Alexey}[1]{\textbf{[#1]}}

 \title{On the Graham--Sloane harmonious labelling conjecture}
\author{Alp M\"uyesser}
\address{(Alp M\"uyesser) New College, University of Oxford, UK.}
\email{alp.muyesser@new.ox.ac.uk}

\author{Alexey Pokrovskiy}
\address{(Alexey Pokrovskiy) University College London, UK.}
\email{dralexeypokrovskiy@gmail.com}

\begin{document}

\begin{abstract}
    Consider an abelian group $G$ of order $n$ and a tree $T$ on $n$ vertices. When is it possible to (bijectively) label $V(T)$ by $G$ so that along all edges $xy$ of $T$, the sums $x+y$ are distinct? This problem can be traced back to the work of Graham and Sloane on the harmonious labelling conjecture, and has been studied extensively since its introduction in 1980. We give a precise characterisation that holds for all bounded degree trees. In particular, our characterisation implies that if $G=\mathbb{Z}/n\mathbb{Z}$ and $T$ is a bounded degree tree, the desired labelling exists. This confirms a conjecture of Graham and Sloane from 1980, and another conjecture of Chang, Hsu, and Rogers from 1987, for bounded degree trees. Our results also have further applications for the study of graph coverings.   
\end{abstract}
\maketitle

\section{Introduction}
Graph labelling problems are some of the most well-studied problems in combinatorics. A prototypical and infamous problem in this direction is the Ringel--Kotzig conjecture, or as more popularly known, the graceful tree conjecture, that asserts that the vertices of any $n$-vertex tree can be (bijectively) labelled by $\{0,1,\ldots, n\}$, so that along all edges of $T$, the absolute value of the difference of the labels of $x$ and $y$ are distinct. We refer the reader to the comprehensive survey of Gallian \cite{labelling} for an overview of the area. 
\par In an influential paper from 1980, motivated by related problems concerning additive bases and error-correcting codes, Graham and Sloane \cite{graham1980additive} studied an analogous labelling problem for trees, where the label set is viewed modulo some integer $n$, and the requirement is that along all edges of the tree, the sum of the labels of $x$ and $y$ are distinct (modulo $n$). More precisely, they conjectured the following.
\begin{conjecture}[The harmonious labelling conjecture, \cite{graham1980additive}] For any $n$-edge tree $T$, there exists a labelling $\phi\colon V(T)\to \mathbb{Z}_n$ so that for all $xy\in E(T)$, $\phi(x)+\phi(y)$ attains a distinct value, and furthermore, $\phi$ uses exactly one label on two vertices. 
\end{conjecture}
This conjecture is known to hold for trees on $\leq 31$ vertices, caterpillars, and many other special classes of trees (see Chapter 2 in \cite{labelling} and the references therein). A relaxation of the conjecture, where the labels come from $\mathbb{Z}_{n+o(n)}$ rather than $\mathbb{Z}_n$, is known to be true due to work of Montgomery, Pokrovskiy, and Sudakov \cite{approximateringel}. This latter result was an important ingredient in the work of the same authors \cite{ringel} settling Ringel's conjecture (see also the independent proof of Keevash and Staden \cite{keevash2025ringel}) which states that the edges of $K_{2n+1}$ can be decomposed into any $n$-edge tree $T$. These last results were originally phrased using the language of rainbow subgraphs which we will also use in this paper. 
\par Recall that an edge-coloured graph is called \textit{rainbow} if each colour appears at most once. Given a graph $T$ and a coloured graph $G$, a \textit{rainbow copy of $T$ in $G$} is a subgraph $T'$ of $G$ which is isomorphic to $T$ and all of whose edges have different colours. There are many conjectures that can be rephrased using this definition \cite{Pokrovskiy2022_RainbowSubgraphs}. For example the  Ringel--Kotzig conjecture mentioned in the first paragraph is equivalent to stating that a rainbow copy of any tree $T$ exists in the colouring of $K_{|T|}$ whose vertices are $1, \dots, |T|$ with $ij$ coloured by $|i-j|$.
\par The harmonious labelling conjecture of Graham and Sloane, on the other hand, is equivalent to a rainbow embedding problem where the edge-colouring rule comes from the Cayley-sum graph of a cyclic group. To make this formal, we give the following definition. \vspace{-1mm}
\begin{definition}
    Let $G$ be an abelian group. We define $K_G$ to be the edge-coloured complete graph on vertex set $G$, where the edge $xy$ is assigned the colour $x+y$. 
\end{definition}
By definition, an embedding of $T$ on $K_G$ gives a rainbow copy whenever the labelling induced by the mapping satisfies that for each edge $xy\in G$, $x+y$ attains a distinct value. 
\par In this paper, we consider embeddings of trees $T$ on $K_G$ where $|V(G)|=|V(T)|$. This informs the harmonious labelling conjecture directly, as it is a more restrictive notion. To see this, let us consider a $n$-edge, $(n+1)$-vertex tree $T$, and let us construct $T'$ by deleting an arbitrary leaf $v$, letting $w$ be the parent of $v$. Suppose we found a rainbow embedding $\phi\colon V(T')\to V(K_{\mathbb{Z}_n})=\mathbb{Z}_n$ of $T'$ in $K_{\mathbb{Z}_n}$. There exists a unique element $c\in \mathbb{Z}_n$ not used as a colour along any edge in the embedding given by $\phi$, so we can then extend $\phi$ to $T$ by defining $\phi(v)=c-\phi(w)$, meaning that $\phi(v)+\phi(w)=c$, which implies that the harmonious labelling conjecture holds for $T$. 
\par This brings us to the following natural conjecture, that if true, would directly imply the harmonious labelling conjecture of Graham and Sloane.
\begin{conjecture}\label{conj:generalcyclic} For all $n$-vertex trees $T$, $K_{\mathbb{Z}_n}$ contains a rainbow copy of $T$. 
\end{conjecture}
Chang, Hsu, and Rogers~\cite{chang1981additive} made an even stronger conjecture in 1987 that $K_{\mathbb{Z}_n}$ contains a rainbow copy of every $T$ where $0$ does not appear as a sum of endpoints along any edge. Conjecture~\ref{conj:generalcyclic} is also a special case of a conjecture of Hovey \cite{hovey1991cordial} on cordial labellings.
\par In this paper, we study a natural generalisation of Conjecture~\ref{conj:generalcyclic} where $\mathbb{Z}_n$ is a general abelian group. A direct extension of Conjecture~\ref{conj:generalcyclic} is not possible in this set-up, due to following well-known construction of Maamoun and Meyniel~\cite{maamoun1984problem} from 1984.
\begin{example}\label{ex:maamoun} Let $T$ be a $n$-vertex path where $n=2^k$ for some $k\geq 2$, and let $G=\mathbb{Z}_2^k$, then $K_G$ contains no rainbow copy of $T$.
\end{example}
\begin{proof}
    Suppose otherwise and take a bijection $\phi\colon V(P_{2^k})\to \mathbb{Z}_2^k$. The set $C$ consisting of sum of labels along edges must be the set $\mathbb{F}_2^k\setminus \{0\}$ because $0$ does not appear as a colour in $K_G$, since $G=\mathbb{Z}_2^k$. Denote by $v,w$ the endpoints of the path $T$. We have $$0=\sum\mathbb{Z}_2^k\setminus \{0\}  =\sum_{xy\in E(T)}\phi(x)+\phi(y) = \sum_{x\in V(T)}\mathrm{deg}_T(x)\cdot \phi(x) =\phi(v)+\phi(w) \neq 0,$$ giving a contradiction. \end{proof}
With a similar proof that we will present formally in Section~\ref{sec:characterisation}, we can also show that if $T$ has a rainbow copy in $K_{\mathbb{F}_2^k}$ (supposing $k\geq 2$), then $T$ cannot have precisely two vertices of even degree. However, these examples barely scratch the surface of the space of possible obstructions. Indeed, there are several other constructions of trees $T$ and abelian groups $G$ where $T$ does not have a rainbow copy in $K_G$. Such constructions are systematically studied in recent work of Jamison and Kinnersley \cite{jamison2022rainbow}, and further results are presented in \cite{cichacz2025rainbow}. The diversity of the constructions present in \cite{jamison2022rainbow, cichacz2025rainbow} indicates that a full characterisation of when a tree admits a rainbow copy in a given $K_G$ is presently out of reach.
\par The contribution of our main result, stated below, is twofold: we find a new class of constructions, and also show that they are the only ones, when $T$ is assumed to be bounded degree. The \textit{characteristic} of an abelian group $G$ is the
smallest positive integer $m$ such that $m\cdot g = 0$ for all $g\in G$.

\begin{theorem}\label{Theorem_main_intro} For any $\Delta$, there exists a $n_0$ sufficiently large so that the following holds for any $n\geq n_0$. Let $T$ be an $n$-vertex tree with $\Delta(T)\leq \Delta$ and $G$ an abelian group of size $n$. There is a rainbow copy of $T$ in $K_G$ if, and only if, we have none of the following:
\begin{enumerate}[(1)]
\item $G=\mathbb{Z}_2^k$ and  $T$ is a path or has precisely two vertices of even degree.
\item $G$ has characteristic $m$, $T$ has adjacent vertices $u$ and $v$ such that $\mathrm{deg}(u)\equiv \mathrm{deg}(v)\equiv 0 $ $(\mathrm{mod}$ $m)$ and furthermore for all $v\in V(T)\setminus\{u,v\}$, $\mathrm{deg}(v)\equiv 1 $ $(\mathrm{mod}$ $m)$.
\item $G=\mathbb{Z}_2^k$, $k\geq 2$, $T$ contains precisely $4$ vertices of  even degree and has a perfect matching when restricted to these $4$ vertices.
\end{enumerate}
\end{theorem}
The construction in (1)  is due to Maamoun and Meyniel \cite{maamoun1984problem} (see also the concluding remarks of~\cite{benzing2020long}), (2) is due to Jamison and Kinnersley \cite{jamison2022rainbow}, and (3), to the best of our knowledge, is a novel example. Just as in Example~\ref{ex:maamoun}, it is not hard to verify that in all three cases, $T$ does not have a rainbow copy in $K_G$. The difficult part of Theorem~\ref{Theorem_main_intro} is the other direction: showing that (for bounded degree trees) these three examples are the only obstructions for the existence of rainbow copies of trees in a given $K_G$. This is fairly surprising, as there is no a priori reason for such a simple characterisation to exist. In fact, for trees of unbounded degree, we were unable to come up with a plausible conjecture of a characterisation which includes all of the examples discovered by Jamison and Kinnersley~\cite{jamison2022rainbow}.
\subsection{Applications}
We now discuss some applications of our main result.
\subsubsection{On the conjectures of Graham--Sloane, Chang--Hsu--Rogers, Andersen and Montgomery--Pokrovskiy--Sudakov}
\par The most direct application of Theorem~\ref{Theorem_main_intro} is that it directly implies Conjecture~\ref{conj:generalcyclic} and therefore, the Graham--Sloane harmonious labelling conjecture, for bounded degree trees.
\begin{corollary}\label{cor:boundeddegreeGS} For every $\Delta\in \mathbb{N}$, there exists some $n_0\in \mathbb{N}$ sufficiently large so that for all $n\geq n_0$, Conjecture~\ref{conj:generalcyclic} holds for all $n$-vertex trees of maximum degree at most $\Delta$. In particular, the Graham--Sloane harmonious tree labelling conjecture also holds in this regime.
\end{corollary}
More generally, in Theorem~\ref{thm:chang} we show that Conjecture~\ref{conj:generalcyclic} holds with the additional requirement that the label $0$ is not used as a colour on any edge, thereby confirming the Chang--Hsu--Rogers conjecture \cite{chang1981additive} (for bounded degree trees). For convenience, we use the notation ``$\alpha\gg \beta$'' to mean ``$\forall \alpha\in (0,1], \exists \beta_0$ such that $\forall\beta\in (0, \beta_0]$ the following holds...'' in the remainder of the paper. So the bounded degree assumption can be compactly packaged as $\Delta\gg n^{-1}$, where $\Delta$, $n$ are the maximum degree and the number of vertices of a tree, respectively.
\par The methods we develop in the paper also have the following consequence stating that the obstructions disappear if the host graph has one additional vertex compared to the tree we are trying to embed.
\begin{theorem}\label{thm:vertexspace}
    Let $\Delta\gg n^{-1}$. Let $G$ be an abelian group of size $n$ and let $T$ be a $(n-1)$-vertex tree. If $T$ has maximum degree at most $\Delta$, then $G$ has a rainbow copy in $K_G$. 
\end{theorem}
 Perhaps unexpectedly, the bounded degree assumption cannot be omitted in Theorem~\ref{thm:vertexspace}, as we will shortly show in Example~\ref{ex:new}. The proof of Theorem~\ref{thm:vertexspace} will be given in the next section as a direct consequence of one of our main theorems, namely, Theorem~\ref{Theorem_embed_tree_prescribed_sets}.
\par We briefly survey some related conjectures in the area to give some context to Theorem~\ref{thm:vertexspace}. A well-known conjecture of Andersen \cite{andersen1989hamilton, alon2017random} states that any properly-coloured $K_n$ contains a rainbow path on $n-1$ vertices. Montgomery, Pokrovskiy, and Sudakov \cite[Conjecture 11.1]{approximateringel} made a related conjecture that any proper colouring of a $K_n$ contains every tree on $n-C$ vertices, where $C$ is an absolute constant. Theorem~\ref{thm:vertexspace} confirms the former conjecture for colourings coming from Cayley-sum graphs, and it confirms the latter conjecture for bounded-degree trees. In fact, under these assumptions, Theorem~\ref{thm:vertexspace} shows $C$ can be taken to be $1$ in the conjecture of Montgomery, Pokrovskiy, and Sudakov. Perhaps surprisingly, such a result is not true under no assumptions on the degrees, meaning that one extra vertex isn't always sufficient, as the following example shows\footnote{This example was discovered by the authors after the first version of the article appeared on the arXiv. We remark that the example disproves Conjectures 1.8 and 6.3 that appeared in the first version of the article.}. 
\begin{example}\label{ex:new}
    Let $G=\mathbb{Z}_2^k$, with $k\geq 3$, and let $T$ be a tree consisting of a path on $3$ edges, with $2^k-5$ additional leaves added incident on one of the endpoints of the path (so $T$ has $2^k-1$ vertices). Then, $K_G$ contains no rainbow copy of $T$. 
\end{example}
\begin{proof}
    Suppose $K_G$ had a rainbow copy of $T$. As $T$ has $2^k-2$ edges, and $K_G$ has $2^k-1$ colours, all but one colour that appears in $K_G$ must be present in the rainbow embedding of $T$. Let $P\subseteq \mathbb{Z}_2^k$ denote the embedding locations of the $3$ vertices of the $3$-edge path, and let $x_0$ denote the vertex of degree $2^k-4$. Let the colours along the $3$-edge path be $c_1,c_2,c_3$, and the vertices be $x_0,x_1,x_2,x_3$ which are all distinct. The key observation is that any colour between $x_0$ and a vertex of $P$ cannot appear on $T$ anywhere other than the edges of the $3$-edge path. The colour of $(x_0,x_2)$ is distinct from any of $c_1,c_2,c_3$ due to $K_G$ being properly coloured, so it cannot appear on $T$. Moreover, the colour of $(x_0,x_3)$ is distinct from $c_1$ and $c_3$ due to $K_G$ being properly coloured, and is distinct from $c_2$ because otherwise we would obtain a $4$-cycle with the colour pattern $c_1,c_2,c_3,c_2$ in $K_G$. As in any cycle the colour patterns sum to $0$ in $K_G$, this implies $c_1=c_3$. So both the colour of $(x_0,x_2)$ and $(x_0,x_3)$ cannot appear anywhere on the rainbow embedding of $T$, giving a contradiction, as all but one colour of $K_G$ must be present in the rainbow embedding of $T$.
\end{proof}

\subsubsection{Orthogonal double covers}
\par Theorem~\ref{Theorem_main_intro} has an interesting connection with the study of orthogonal double covers. An orthogonal double cover of a complete graph $K_n$ is a collection of isomorphic subgraphs $G_1,\ldots, G_k$ where for each $i,j\in [k]$, $G_i$ and $G_j$ share exactly one edge, and furthermore, each edge of $K_n$ is included in exactly two of the subgraphs. The motivation for investigating the existence of such objects comes from statistical design theory, see \cite[Chapter 2]{dinitz1992contemporary}. An important conjecture in the area, due to Gronau, Mullin, and Rosa \cite{gronau1997orthogonal} from 1997, is the following. 
\begin{conjecture}
    For any $n$-vertex tree that is not a path on $4$ vertices, $K_n$ has an orthogonal double cover by copies of $T$. 
\end{conjecture}
The above conjecture is known to hold for $n\leq 13$, stars, and trees with diameter $\leq 3$ \cite{gronau1997orthogonal, leck1997orthogonal}. The conjecture is also known to hold ``approximately'' whenever $n$ is a power of two \cite{approximateringel}. The following observation from \cite{approximateringel} makes the connection between orthogonal double covers and rainbow subgraphs explicit. 
\begin{observation}\label{obs:cyclicshift}
    Let $T$ be a tree on $2^k$ vertices and suppose that $T$ has a rainbow copy in $K_{\mathbb{Z}_2^k}$. Then, $K_{2^k}$ admits an orthogonal double cover by copies of $T$.
\end{observation}
\begin{proof}
    For some $x\in \mathbb{Z}_2^k$ and a tree $T$, we denote by $x+T$ the isomorphic tree obtained by having the vertex $v$ of the tree $T$ map to $x+v$. We call $x+T$ a \textit{translate} of $T$. Suppose now that $T$ is a $2^k$-vertex rainbow tree of  $K_{\mathbb{Z}_2^k}$. Observe that each translate also yields a rainbow copy of the tree $T$, as any edge of $T$ of colour $c$ again maps to an edge of colour $c$, since $x+x=0$ over $\mathbb{Z}_{2}^k$. We claim that the $2^{k}$ translates of $T$ give an orthogonal double cover of (the uncoloured version of) $K_{\mathbb{Z}_2^k}$, i.e. $K_{2^k}$. Firstly, note that for distinct $x$ and $y$, $x+T$ and $y+T$ meet on exactly one edge. Indeed, as the colour of each edge is preserved under translations, $x+T$ and $y+T$ can only meet if an edge is fixed when $x+T$ is translated by $y-x$. There is precisely one such edge of $x+T$, namely, the one with colour $y-x$. 
    \par Similarly, each edge of $K_{\mathbb{Z}_2^k}$ is an element of two distinct translates $x+T$ and $y+T$. Indeed, suppose an edge $ab$ has colour $c=a+b$, and let $vw$ be the $c=v+w$ coloured edge of $T$, recalling that this is the only edge whose translates can cover $xy$, and we have that $a+b+v+w=c+c=0$. The translations $x=v+a$ and $y=v+b$ map the edge $vw$ to $ab$, as desired.
\end{proof}
Therefore, Theorem~\ref{Theorem_main_intro} gives the following corollary.
\begin{corollary}\label{cor:gronau} Let $\Delta\gg n^{-1}$, and suppose $n$ is a power of two. Suppose $T$ has maximum degree at most $\Delta$. Unless $T$ is one of the examples described in Theorem~\ref{Theorem_main_intro}, then the Gronau--Mullin--Rosa conjecture holds for $T$.
\end{corollary}
\section{Overview and the core lemma}
\textbf{Organisation of the paper.} In this section, we state the main technical contribution of the paper, i.e. Theorem~\ref{Theorem_embed_tree_prescribed_sets}, that gives a simple necessary and sufficient condition for when a bounded degree tree has a rainbow copy inside $K_G$ for an abelian group $G$. This theorem will be proved in Section~\ref{sec:corelemma}, after having introduced some preliminaries and auxiliary results in Section~\ref{sec:preliminaries}. In this section, we give a proof of Theorem~\ref{thm:vertexspace} by assuming Theorem~\ref{Theorem_embed_tree_prescribed_sets}. We also use Theorem~\ref{Theorem_embed_tree_prescribed_sets} to give a direct proof of Corollary~\ref{cor:boundeddegreeGS} (by way of proving the more general Theorem~\ref{thm:chang} that confirms the Chang--Hsu--Rogers and Graham--Sloane conjectures for bounded degree trees). These simple applications of Theorem~\ref{Theorem_embed_tree_prescribed_sets} serve as a warm-up to read the rest of the paper, especially the more refined arguments in Section~\ref{sec:characterisation}. In Section~\ref{sec:characterisation}, we analyse when the necessary and sufficient condition from Theorem~\ref{Theorem_embed_tree_prescribed_sets} can be fulfilled depending on the structure of the abelian group $G$. The result of this analysis, when combined with Theorem~\ref{Theorem_embed_tree_prescribed_sets}, will be the characterisation given in our main result, Theorem~\ref{Theorem_main_intro}.
\begin{center}
\rule{5cm}{0.4pt}
\end{center}
\par Our key result, Theorem~\ref{Theorem_embed_tree_prescribed_sets}, concerns the core of a tree, which we now formally define. Intuitively, the core of a tree is a minimal representative sample from the degree sequence of the tree.
\begin{definition}
Let $T$ be a tree.
We say that an induced subforest  $T_{core}$ is a core of $T$ if for  every $d\leq \Delta(T)$, we have at least one of
\begin{enumerate}[(I)]
\item $T_{core}$ contains all vertices $v$ with  $d_T(v)=d$. 
\item $T_{core}$ and $T\setminus V(T_{core})$  both contain at least 6 vertices $v$ with $d_T(v)=d$.
\end{enumerate} 
\end{definition}
When (I)   occurs we say that ``degree $d$ vertices are exhausted by $T_{core}$''. 
\par Embedding cores of trees is a significantly easier task for bounded degree trees because every bounded degree tree has a small core, as shown below.
\begin{observation}\label{Observation_small_core}
Every tree has a core $T_{core}$ of order $\leq 12\Delta(T)$.  If $\Delta(T)\leq \sqrt{|V(T)}/12$, then there is a core of order exactly $12\Delta(T)$.
\end{observation}
\begin{proof}
For every degree $d$, if there are $\leq 12$ vertices of degree $d$, include all degree $d$ vertices in $T_{core}$, otherwise include exactly 6 degree $d$ vertices in $T_{core}$. This gives a core of size $\leq 12\Delta(T)$. To get a core of size exactly $12\Delta(T)$, add $12\Delta(T)-|V(T_{core})|$ vertices of the most popular degree to $T_{core}$, noting that there are at least $|V(T)|/\Delta(T)\geq 144\Delta(T)$ vertices of this degree (and hence after we add $\leq 12\Delta(T)$ of them to the core, there will still remain 6 outside the core).
\end{proof}
It is worth remarking that cores of trees are not unique.
We can now state a key technical result of the paper. 
\begin{theorem}\label{Theorem_embed_tree_prescribed_sets}
Let $\Delta^{-1}\gg\mu\gg n^{-1}$.
Let $G$ be a group and $T$ a bounded degree tree with $\Delta(T)=\Delta$. Let $V_{target},C_{target}\subseteq G$ with $|T|=|V_{target}|=|C_{target}|+1\geq (1-n^{-\mu})n$. In the case $G=\mathbb{Z}_2^k$, assume $\id \not\in C_{target}$. Let $T_{core}$ be a core of $T$ of size $\leq n^{1-\mu}$. The following are equivalent.
\begin{enumerate}[(i)]
\item  $T$ has a rainbow  embedding $f$ into $(V_{target},C_{target})$.
\item There is a rainbow embedding $\phi$ of $T_{core}$ into $(V_{target},C_{target})$ with $\sum_{v\in V(T_{core})}d_T(v)\phi(v)=\sum C_{target}$ and $\sum \phi(V(T_{core}))=\sum V_{target}$.
\end{enumerate} 
\end{theorem}
Although Theorem~\ref{Theorem_embed_tree_prescribed_sets} gives an equivalence, we only ever use the ``(ii)$\Rightarrow$(i)'' direction of the above result. The proof of the ``(i)$\Rightarrow$(ii)'' direction is easier and can be derived from first principles. We give the details of this direction as well, as we believe the equivalence is of theoretical interest. 
Theorem~\ref{Theorem_embed_tree_prescribed_sets} has an interesting complexity-theoretic corollary. It gives a polynomial-time algorithm to decide whether a bounded degree tree has a rainbow embedding using vertices $V_{target}, C_{target}$ with $|T|=|V_{target}|=|C_{target}|+1\geq (1-n^{-\mu})n$ (where the running time is a polynomial in $n$ whose degree depends on $\Delta, \mu$). The algorithm consists of first finding a core $T_{core}$ of order $12\Delta(T)$ (following the proof of Observation~\ref{Observation_small_core}), and afterwards checking all embeddings of $T_{core}$ into $V_{target}, C_{target}$ to see if they satisfy (ii).

Theorem~\ref{Theorem_embed_tree_prescribed_sets} is proved in Section~\ref{sec:corelemma}. Deriving Theorem~\ref{Theorem_main_intro} from Theorem~\ref{Theorem_embed_tree_prescribed_sets} requires a careful analysis of the structure of bounded degree trees and abelian groups, and is presented in Section~\ref{sec:characterisation}. Theorem~\ref{thm:vertexspace}, on the other hand, follows directly from Theorem~\ref{Theorem_embed_tree_prescribed_sets}.

\begin{proof}[Proof of Theorem~\ref{thm:vertexspace} via Theorem~\ref{Theorem_embed_tree_prescribed_sets}] Let $T$ be a $(n-1)$-vertex bounded degree tree, and let $G$ be a $n$-element abelian group. Let $T_{core}$ be a core of size $12\Delta(T)$ which exists by Observation~\ref{Observation_small_core}. Our goal is to embed $T_{core}$ to $K_G$ via $\phi$ in a rainbow manner, and designate a $C_{target}\supseteq C(\phi(T_{core}))$ of size $n-2$ and designate a $V_{target}\supseteq C(\phi(T_{core}))$ of size $n-1$ so that $\phi$, $V_{target}$, and $C_{target}$ satisfy condition (ii) of Theorem~\ref{Theorem_embed_tree_prescribed_sets}. This amounts to finding a rainbow embedding $\phi$ of $T_{core}$ to $K_G$ satisfying the following two properties. 
\begin{enumerate}
\item Setting $v^*:=\sum G - \sum \phi(V(T_{core}))$, we have that $v^*\notin \phi(V(T_{core}))$.
\item Setting $c^*:=\sum G - \sum_{v\in V(T_{core})}d_T(v)\phi(v)$, we have that $c^*\notin \phi(C(T_{core}))\cup\{0\}$. 
\end{enumerate}
Indeed, we can then select $V_{target}:=G\setminus\{v^*\}$ and $C_{target}:=G\setminus\{c^*, 0\}$ to satisfy the constraints in (ii), and also have that $0\notin C_{target}$ as required by Theorem~\ref{Theorem_embed_tree_prescribed_sets}. So our goal in the remainder of the proof is to find a rainbow embedding $\phi$ of $T_{core}$ satisfying (1) and (2). 
\par Fix two leaves $\ell,w$ of $T$ such that $\ell,w\in T_{core}$ (follows by $d=1$ case of the definition of a core), and if $\ell$ has a neighbour in $T_{core}$, call it $s$, noting $s\neq w$. We will first partially define the embedding $\phi$ on $T-\ell$, and then extend the embedding $\phi$ to all of $T$. 
\par Define $\mathcal{B}$ to consist of the $g\in G$ such that there are at least $n/10$ distinct $x$ such that $x+x=g$. Note $|\mathcal{B}|\leq 10$. We will have the following requirements on $\phi$ (defined on $T-\ell$ for now).
\begin{enumerate}[(A)]
\item $\sum G- \sum \phi(V(T_{core}-\ell))\notin \mathcal{B}$. 

\item If $s$ exists, we require that $ -\phi(s)+ \sum G - \sum_{v\in V(T_{core})\setminus \{\ell\}}d_T(v)\phi(v)\notin \mathcal{B}$. 
\end{enumerate}
To see such a $\phi$ exists, first define a rainbow embedding $\phi$ on ${T-\ell-w}$, which exists by following a greedy algorithm as $12\Delta\ll n$. There are at most $10$ choices of $\phi(w)$ that contradict (A), and at most $10$ choices of $\phi(w)$ that contradict (B), as $\phi(w)$ appears with a coefficient of $1$ in each indexed sum. Since there are at least $n-24\Delta$ possible choices for $\phi(w)$ that yield a rainbow embedding $\phi$, we may find a $\phi$ (defined on $T-\ell$) with properties (A) and (B).
\par It remains to extend $\phi$ by defining $\phi(\ell)$ so that (1) and (2) are satisfied. There are at least $n-24\Delta$ choices for $\phi(\ell)$ that produce a rainbow embedding. We now count the bad choices for $\phi(\ell)$ that would violate one of our two desired conditions.
\par First, there are at most $12\Delta$ choices for $\phi(\ell)$ that yield $\sum G - \sum \phi(V(T_{core}))=v^*\in \phi(V(T_{core}-\ell))$, as $V(\phi(T_{core}-\ell))$ is a fixed set of size $\leq 12\Delta$ and $\sum G - \sum \phi(V(T_{core}))$ takes on a different value for each different choice of $\phi(\ell)$ (as the coefficient is $1$). Let us now count the number of choices of $\phi(\ell)$ that yield $\sum G - \sum \phi(V(T_{core}))=\phi(\ell)$, or equivalently $2\cdot \phi(\ell) = \sum G - \sum \phi(V(T_{core}-\ell))$. Recall that the right hand side of the last equality cannot be in $\mathcal{B}$ by (A). This means that in total, there are at most $n/10 + 12\Delta$ choices of $\phi(\ell)$ that violate (1).
\par We now turn our attention to condition (2). Similar to before, there are at most $12\Delta+1$ choices for $\phi(\ell)$ that yield $c^*\in \phi(C(T_{core}-\ell))\cup\{0\}$. Any other potential conflicts arise from the edge incident on $\ell$ within $T_{core}$, in which case we may suppose that $s$ exists, and thus that the colour of the potential conflict edge is $\phi(s)+\phi(\ell)$. We now have to count the number of choices for $\phi(\ell)$ that yield $\phi(s)+\phi(\ell)=c^*=\sum G - \sum_{v\in V(T_{core})}d_T(v)\phi(v)$, or equivalently, 
$$2\cdot \phi(\ell)=-\phi(s)+\sum G- \sum_{v\in V(T_{core})\setminus\{\ell\}}d_T(v)\phi(v) .$$
\par (B) implies that there are at most $n/10$ bad choices of $\phi(\ell)$, as the right hand side is not in $\mathcal{B}$. In total, this gives $\leq n/5 + 24\Delta +1 $ bad choices for $\phi(\ell)$, so a good choice among the space of $n-24\Delta$ available choices must exist, concluding the proof. \end{proof}

We can also easily prove the Chang-Hsu-Rogers conjecture \cite{chang1981additive} in the bounded degree case. Note that this in particular implies Corollary~\ref{cor:boundeddegreeGS}, that is, the Graham--Sloane conjecture holds for bounded degree trees.
\begin{theorem}\label{thm:chang}
Let $n\gg \Delta^{-1}$  and let $T$ be an $n$-vertex tree with $\Delta(T)\le \Delta$. Then $K_{\mathbb{Z}_n}$ has a rainbow copy of $T$ using colours $\mathbb{Z}_n\setminus\{0\}$.
\end{theorem}
\begin{proof}
 Use Observation~\ref{Observation_small_core} to find a core $T_{core}$ of $T$ with $|T_{core}|\le 12\Delta (T)$.  Let $u$, $v$ be two leaves in $T_{core}$. 

 We claim that there exists a rainbow embedding $\phi$ of $T_{core}\setminus \{u,v\}$ with $\sum_{x\in V(T_{core})\setminus \{u,v\}} (d_T(x)-1)\phi(x)=0$. To see this, first set $m:=|T_{core}\setminus \{u,v\}|$ and  observe that there are at least $n^{m-1}$ functions $\phi:V(T_{core})\setminus \{u,v\}\to \mathbb{Z}_n$ satisfying $\sum_{x\in V(T')} (d_T(x)-1)\phi(x)=0$.  Indeed, each such $\phi$ corresponds to an element of the kernel of a group homomorphism from $(\mathbb{Z}_n)^m$ to $\mathbb{Z}_n$, and the kernel of such a map has size at least $n^{m-1}$ by the first isomorphism theorem (as the image size is trivially at most $n$). If $m=1$ then there can't be any vertex or colour repetitions so any such function is a rainbow embedding. If $m\geq 2$,  then it is easy to check that the number of such functions $\phi$ that aren't injective is $\le \binom m2 n^{m-2}=o(n^{m-1})$ and the number of such functions which repeat colours on edges is $\le \binom {e(T_{core}\setminus \{u,v\})}{2}n^{m-2}\le \binom m2 n^{m-2}=o(n^{m-1})$. Thus there remains some function $\phi$ for which neither happens. 

\par Next, pick $\phi(u), \phi(v)$ so that $\phi$ is a rainbow embedding on all of $T_{core}$ and additionally $\phi(u)+\phi(v)=\sum \mathbb{Z}_n- \sum \phi(V(T_{core})\setminus\{u,v\})$. This is possible because there are exactly $n$ choices of the pair $(\phi(u), \phi(v))$ with  $\phi(u)+\phi(v)=\sum \mathbb{Z}_n- \sum \phi(V(T_{core})\setminus\{u,v\})$, and  $O(|T_{core}|)$ of these do not give a rainbow embedding of $T_{core}$. 

Now set $V_{target}= \mathbb{Z}_n$, $C_{target}= \mathbb{Z}_n\setminus 0$. 
We have $\sum_{x\in V(T_{core})}\phi(x)= \sum\mathbb{Z}_n$ by the previous paragraph. 
Also $\sum_{x\in V(T_{core})}d_T(x)\phi(x)=\sum_{x\in V(T_{core})\setminus \{u,v\}}(d_T(x)-1)\phi(x) +\sum\phi(V(T_{core}))=0+\mathbb{Z}_n.$ Thus Theorem~\ref{Theorem_embed_tree_prescribed_sets} applies to give a rainbow embedding of $T$ using $V_{target}$, $C_{target}$, proving the theorem.
\end{proof}

\section{Preliminaries}\label{sec:preliminaries}
For a subset $S$ of an abelian group, define $S-S=\{x-y:x,y\in S\}$.
We use ``$\alpha\gg \beta$'' to mean ``$\forall \alpha\in (0,1], \exists \beta_0$ such that $\forall\beta\in (0, \beta_0]$ the following holds...''. When we write something like $\alpha \gg n^{-1}$, we implicitly also require that the inverted terms like $n$ are positive integers.

We'll often use that ``$\alpha\gg \beta\gg \gamma$ implies that $\beta \geq \gamma/\alpha$''. Indeed, after unpacking the definition of ``$\gg$'', the statement becomes ``$\forall \alpha>0, \exists \beta_0$ such that $\forall\beta\in(0, \beta_0], \exists \gamma_0$ such that $\forall \gamma\in (0, \gamma_0]$ we have $\beta \geq \gamma/\alpha$''. This is true by picking $\beta_0=\alpha$ and $\gamma_0=\beta^2$. Then $\gamma\leq \gamma_0\leq \beta^2\leq \beta \beta_0=\beta\alpha$ which is equivalent to $\beta \geq \gamma/\alpha$.

For a graph $G$, and two sets of vertices $A,B$, we use $G[A,B]$ to denote the subgraph on $A\cup B$ consisting of all edges with one endpoint in $A$ and one endpoint in $B$.

\subsection{Completion lemmas}\label{Section_completion_lemmas}
The following theorem is referred to as the ``random Hall-Paige conjecture'' and is formulated and proved by the authors in~\cite{muyesser2022random}.
\begin{theorem}[\cite{muyesser2022random}, Theorem 4.6]\label{thm:maintheoremsemidisjoint} Let $p\geq n^{-1/10^{102}}$. Let $G$ be a group of order $n$. Let $R^1,R^2\subseteq G$ be disjoint $p$-random subsets, and let $R^3\subseteq G$ be a $p$-random subset, sampled independently with $R^1$ and $R^2$. Then, with high probability, the following holds. 
\par Let $X,Y,Z$ be subsets of $G_A$, $G_B$, and $G_C$ be equal sized subsets satisfying the following properties.
\begin{itemize}
    \item $|(R^1_A\cup R^2_B\cup R^3_C) \triangle (X\cup Y\cup Z) |\leq p^{10^{18}}n/\log(n)^{10^{18}}$
    \item $\sum X+\sum Y - \sum Z = 0$ (in the abelianization of $G$)
    \item If $G=\mathbb{Z}_2^k$ for some $k$, suppose that $\id \notin Z$.
\end{itemize}
Then, $K_G$ contains a perfect $Z$-matching from $X$ to $Y$.
\end{theorem}
The following lemma is also from~\cite{muyesser2022random} and proved by combining the above result with the sorting network method.

\begin{lemma}[\cite{muyesser2022random}, Lemma 6.21]\label{lem:pathlikemain}
    Let $1/n\ll \gamma, p\leq 1$, let $t$ be a positive integer between $(\log n)^7$ and $(\log n)^8$, and let $q$ satisfy  $p=(t-1)q$. Let $G$ be a group of order $n$. Let $V_{str}, V_{mid}, V_{end}$ be disjoint random subsets with $V_{str}, V_{end}$ $q$-random and $V_{mid}$ $p$-random. Let $C$ be a $(q+p)$-random subset, sampled independently with the previous sets. Then, with high probability, the following holds. 
    \par Let $V_{str}'$, $V_{end}'$, $V_{mid}'$ be disjoint subsets of $G$, let $C'$ be a subset of $G$, and let $\ell=|V'_{mid}|/(t-1)$. Suppose all of the following hold.
    \begin{enumerate}
        \item For each random set $R\in\{V_{str}, V_{mid}, V_{end}, C\}$, we have that $|R\Delta R'|\leq n^{1-\gamma}$.
        \item $\sum V_{end}'+\sum V_{str}'+ \sum V'_{mid}+ \sum V'_{mid}=\sum C'$ holds in the abelianization of $G$.
        \item $\id \notin C'$ if $G$ is an elementary abelian $2$-group.
        \item $\ell:=|V_{str}'|=|V_{end}'|=|V'_{mid}|/(t-1)=|C'|/t$
    \end{enumerate}
    Then, given any bijection $f\colon V_{str}'\to V_{end}'$, we have that $K_G[V_{str}'\cup V_{end}'\cup V_{mid}';C']$ has a rainbow $\vec{P}_t$-factor where each path starts on some $v\in V_{str}'$ and ends on $f(v)\in V_{end}'$.
\end{lemma}

\subsection{Embeddings and injections}
An embedding of a graph $T$ into another graph $K$ is an injection $f:V(T)\to V(K)$ which maps edges to edges. 
A rainbow embedding of a graph $T$ into a coloured complete graph $K$ is an injection $f:V(T)\to V(G)$ with the property that the colours of all edges $f(u)f(v)$ are distinct for edges $uv\in E(T)$. For a vertex set $V_{target}\subseteq V(G)$ and colour set $C_{target}\subseteq V(G)$, we say that $f$ is an embedding into $(V_{target}, C_{target})$ if $V(f(T))\subseteq V_{target}$ and $C(f(T))\subseteq C_{target}$.

Our proof works by building rainbow embeddings gradually. Not all embeddings we construct along the way are rainbow. An important concept we need is of a ``pseudoembedding'' into $(V_{target}, C_{target})$ --- informally this is a (not-necessarily rainbow) embedding of $T$ into $K$ with the property that the sums of the vertices/colours of the embedding are the same as what they would be if the embedding was rainbow.
\begin{definition}\label{Definition_pseudoembedding}
Let $V_{target},C_{target}$ be sets of vertices/colours in $K_G$ for a group $G$.
We say that $f:V(T)\to G$ is a pseudoembedding of $T$ into  $(V_{target},C_{target})$ if $f$ is an injection with  $im|_f=V_{target}$ and $\sum_{v\in V(T)} d_T(v)f(v)=\sum C_{target}$. 
\end{definition}
Note that $C_{target}$ plays very little role in this definition. However when we apply it, we will have $|C_{target}|=|E(T)|$ as if we were trying to find an actual rainbow embedding of $T$. 
Our proofs are probabilistic --- meaning that we work with random embeddings of graphs. We need the notion of a random embedding having a ``nice distribution''.
\begin{definition}
Let $T$ be a graph, $G$ a group,  and $f:V(T)\to K_G$ be a random function. 
\begin{itemize}
\item For $U\subseteq V(T)$, we say that $f$ is $\varepsilon$-uniform on $U$ if there is an $(|U|/|G|)$-random subset $U_{\mathrm{rand}}\subseteq V(K_G)$ with $|f(U)\Delta U_{\mathrm{rand}}|\leq \varepsilon n$ with probability $1-o(n^{-1})$. 
For $U\subseteq E(T)$, we say that $f$ is $\varepsilon$-uniform on $U$ if there is an $(|U|/|G|)$-random subset $U_{\mathrm{rand}}\subseteq C(K_G)$ with $|C(f(U))\Delta U_{\mathrm{rand}}|\leq \varepsilon n$ with probability $1-o(n^{-1})$. 

\item  Given disjoint vertex sets $U_1, U_2\subseteq V(T)$, say that $f$ is $\varepsilon$-uniform on $\{V(T), U_1, U_2, E(T)\}$ if  it is $\varepsilon$-uniform on each of $V(T), U_1, U_2, E(T)$, and additionally if  $V^{\mathrm{rand}}, U_1^{\mathrm{rand}}, U_2^{\mathrm{rand}}, E^{\mathrm{rand}}$ are the sets witnessing this, then the joint distribution on the $U_1^{\mathrm{rand}}, U_2^{\mathrm{rand}}, V^{\mathrm{rand}}\setminus (U_1^{\mathrm{rand}}\cup U_2^{\mathrm{rand}})$ is that of disjoint random sets, and also  $V^{\mathrm{rand}}, U_1^{\mathrm{rand}}, U_2^{\mathrm{rand}} $  are independent of the $E^{\mathrm{rand}}$. 

\end{itemize}
\end{definition}

The following lemma says that all almost-spanning trees can be approximately embedded into properly coloured complete graphs --- and that this can be achieved by a random embedding which is uniform on prescribed sets. The statement and proof are essentially the same as similar results from~\cite{approximateringel, ringel}. However, for completeness, we give a proof of it in the appendix.
\begin{lemma}\label{Lemma_approximate_embedding}
 Let $\Delta^{-1}\gg \varepsilon,\delta \gg n^{-1}$.
Let $K_{n}$ be properly $n$-edge-coloured and  $T$ a  forest with with $\Delta(T)\leq \Delta$ and $|T|\leq (1-n^{-\delta})n$, and suppose we have a partition $V(T)=U_1\cup U_2\cup U_3$. 
Then there is a random  $f:V(T)\to K_n$ which is $n^{-\varepsilon}$-uniform on $\{V(T), U_1, U_2, E(T)\}$ 
\end{lemma}

\subsection{Approximations of trees}
We use the results from Section~\ref{Section_completion_lemmas} to turn almost-spanning embeddings of trees into spanning ones. However, to do this, the almost-spanning tree that we work with must have certain properties. We call a subtree with these properties an ``approximation'' of a tree. 
\begin{definition}
For a tree, set $t(T):=\lceil 2\log^7 |V(T)|\rceil$. Let $T_{appr}$ be an induced subtree of $T$.
\begin{itemize}
\item $T_{appr}$ is a matching-approximation of $T$ if $T\setminus E(T_{appr})$ is a matching $M$ of even size  $\geq |V(T)|/20\Delta(T) t(T)$. Let $U(T_{appr}):=V(M)\cap V(T_{appr})$ to get a set of size exactly $e(M)$, and split $U(T_{appr})$ into two subsets $U_1(T_{appr}), U_2(T_{appr})$ of the same size. 
\item $T_{appr}$ is a path-approximation of $T$ if $T\setminus E(T_{appr})$ is a collection of $\geq |V(T)|/20t(T)$ vertex-disjoint paths of length $t(T)$. Let $U(T_{appr})$ be the set of endpoints of these paths. Orienting each path arbitrarily, let $U_1(T_{appr})$ be the set of starts of these paths, and $U_2(T_{appr})$ be the set of ends.  
\end{itemize}
In both cases set $L(T_{appr}):=V(T)\setminus V(T_{appr})$.
We say that $T_{appr}$ is an approximation of $T$ if it is either a path-approximation or a matching-approximation.
\end{definition}
Note that in both cases, we have $|U_1(T_{appr})|=|U_2(T_{appr})|$ and $U(T_{appr})=U_1(T_{appr})\cup U_2(T_{appr})$. In both cases, define $p(T_{appr},n):=|L(T_{appr})|/n$, $q(T_{appr},n):=|U_1(T_{appr})|/n=|U_2(T_{appr})|/n$,  $r(T_{appr},n):=|E(T)\setminus E(T_{appr})|/n$.
The following lemma is standard.
\begin{lemma}[\cite{krivelevich2010embedding}, Lemma 2.1]\label{Lemma_tree_leaves_bare_paths}
Let $t\in \mathbb N$. Every tree either has  $|V(T)|/10t$ leaves (and hence a matching of $|V(T)|/10\Delta(T)t$ leaves), or has $|V(T)|/10t$ disjoint bare paths of length $t$.
\end{lemma}
A consequence of this is that every tree either has a matching of leaves of size $n/ 10 \Delta(T)\lceil 2\log^7n\rceil$ or has $n/10\lceil2\log^7n\rceil$ disjoint bare paths of length $\lceil2\log^7n\rceil$ --- and hence each tree either has a matching-approximation or a path-approximation. We'll need the following version of this which also makes sure that non-exhausted degrees have vertices inside the approximation.

\begin{lemma}\label{Lemma_nonexhausted_vertices_in_Tabs}
Let $\Delta^{-1}\gg \alpha\gg n^{-1}$. Let $T$ be a $n$-vertex tree with $\Delta(T)\leq \Delta$ and $T_{core}$ a core of $T$ of order $\leq n^{1-\alpha}$.
Then there is an approximation $T_{appr}$ of $T$ with $V(T_{core})\subseteq V(T_{appr})$ such that for each non-exhausted degree $d$ of $T_{core}$, there are at least $6$ vertices of degree $d$ in $V(T_{appr})\setminus V(T_{core})$. 
\end{lemma}
\begin{proof}
By Lemma~\ref{Lemma_tree_leaves_bare_paths}, $T$ either has a matching   of $n/10\Delta(T)t(T)$ leaves or a set of $n/10t(T)$ disjoint bare paths of length $t(T)$.  Since $n/60\Delta(T)t(T)\geq n^{1-\alpha}\geq |V(T_{core})|$, there are either $2\lceil n/40\Delta(T) t(T)\rceil< \frac 12 n/10\Delta(T)t(T)$ leaves disjoint from $V(T_{core})$ or a set of $2\lceil n/40t(T)\rceil< \frac 12 n/10t(T)$ disjoint bare paths of length $t$ disjoint from $V(T_{core})$.
Deleting these gives either a matching-approximation or path-approximation of $T$ which contains all the vertices of $T_{core}$. 

For the ``such that'' part, note  that all $v\in L(T_{appr})=V(T)\setminus V(T_{appr})$  have the same degree $k:=1$ or $2$, that $T$ has $\geq n/10\Delta(T)t(T)$ vertices of this degree,  less than half of these are outside $T_{appr}$, and so $T_{appr}\setminus T_{core}$ contains at least $n/20\Delta(T)t(T)-|T_{core}|\geq n/20\Delta(T)t(T)-n^{1-\alpha}\geq 6$ vertices of this degree. For other non-exhausted degrees $d$, $T_{appr}\setminus T_{core}$ contains all the vertices of $T\setminus T_{core}$   of degree $d$ and hence contains $\geq 6$ vertices of  degree $d$ by the definition of ``core''.
\end{proof}

The following two observations are immediate by plugging in the definitions of $p(T,n), r(T,n), q(T,n)$ and rearranging. 
\begin{observation}\label{Observation_matching-like}
Let $\Delta^{-1}\gg n^{-1}$, and let $T$ be a tree with $\Delta(T)\leq \Delta$, $|V(T)|\in[n/2, n]$ and $T_{appr}$ a matching-approximation of $T$. Then we have the following.
\begin{enumerate}[(i)]
\item $p(T_{appr},n)=2q(T_{appr},n)$.
\item $|V(T_{appr})|/n=1-p(T_{appr},n) +\frac{|V(T)|-n}n$ and $|E(T_{appr})|/n=1-r(T_{appr},n) +\frac{|V(T)|-1-n}n$
\item $r(T_{appr},n)=p(T_{appr},n)$. 
\item $p(T_{appr},n),r(T_{appr},n),q(T_{appr},n)\geq 1/100\Delta\log^7 n\geq n^{-1/100}$.
\end{enumerate}
\end{observation}

\begin{observation}\label{Observation_path-like}
Let $1\gg n^{-1}$, and let $T$ be a  tree with $|V(T)|\in[n/2, n]$ and $T_{appr}$ a path-approximation of $T$. Then we have the following.
\begin{enumerate}[(i)]
\item  $p(T_{appr},n)=(t(T)-1)q(T_{appr},n)$.
\item $|V(T_{appr})|/n=1-p(T_{appr},n) +\frac{|V(T)|-n}n$ and $|E(T_{appr})|/n=1-r(T_{appr},n) +\frac{|V(T)|-1-n}n$ 
\item $r(T_{appr},n)=p(T_{appr},n)+q(T_{appr},n)$.
\item $p(T_{appr},n), r(T_{appr},n) \geq 0.01$  and $q(T_{appr},n)\geq n^{-1/100}$.
\item $|U_1(T_{appr})|=|U_2(T_{appr})|=\frac{|E(T)|-|E(T_{appr})|}{t(T)}=\frac{|V(T)|-|V(T_{appr})|}{t(T)-1}$
\end{enumerate}
\end{observation}

\section{Tree embeddings using prescribed vertices and colours}\label{sec:corelemma}
The goal of this section is to prove Theorem~\ref{Theorem_embed_tree_prescribed_sets}, which characterizes when a large bounded degree tree has a rainbow embedding into $K_G$ using prescribed vertices and colours. The embedding proceeds in several stages, and we start by proving the lemma used for the very last part of the embedding. The following lemma takes a rainbow embedding $f$ of an approximation of a tree $T$ and extends it to an embedding $f'$ of all of $T$. The key additional condition that we need from $f$ is that $f$ extends to a pseudoembedding (as in Definition~\ref{Definition_pseudoembedding}). The proof relies crucially on Theorem~\ref{thm:maintheoremsemidisjoint} and Lemma~\ref{lem:pathlikemain}. 
\begin{lemma}\label{Lemma_final_extension}
Let $\Delta^{-1}\gg \mu,\alpha\gg n^{-1}$.
Let $G$ be an abelian group and let $T$ be a tree with $|V(T)|\geq(1-n^{-\alpha})n$ and $\Delta(T)\leq \Delta$.
Let $V_{target},C_{target}\subseteq G$ with $|V_{target}|=|C_{target}|+1=|V(T)|$. When $G=\mathbb{Z}_2^m$, additionally assume that $e\not\in C_{target}$. Let $T_{appr}$ be an approximation of $T$.

Let $f:V(T_{appr})\to V(K_G)$ be a random function satisfying the following:
\begin{itemize}
\item With high probability, $f$ is a rainbow embedding of $T_{appr}$ into $(V_{target}, C_{target})$.
\item $f$ is $n^{-\mu}$-uniform on $\{V(T_{appr}), U_1(T_{appr}), U_2(T_{appr}), E(T_{appr})\}$ 
\item With high probability, $f$ is extendable to a pseudoembedding $h$ into $(V_{target}, C_{target})$.
\end{itemize}
 Then, there is a random $f':V(T)\to  V(K_G)$ which extends $f$ and is a rainbow embedding into $(V_{target},C_{target})$, with high probability.
\end{lemma}
\begin{proof}
Let $\Delta^{-1}\gg \mu,\alpha\gg \gamma\gg n^{-1}$. Fix $t:=t(T)$, $p:=p(T_{appr},n)$, $q:=q(T_{appr},n)$, $r:=r(T_{appr},n)$.
Let $V^{\mathrm{rand}}, C^{\mathrm{rand}}, U_1^{\mathrm{rand}}, U_2^{\mathrm{rand}}$ be $|V(T_{appr})|/n,|E(T_{appr})|/n,q,q$-random sets produced by $n^{-\mu}$-uniformity of $f$.  Let $V^c_{\mathrm{rand}}=V(K_G)\setminus V^{\mathrm{rand}}, C^c_{\mathrm{rand}}=C(K_G)\setminus C^{\mathrm{rand}}$, noting that these are $p'$-random for $p'=1-|V(T_{appr})|/n\in [p, p+2n^{-\alpha}]$ and $r'$-random for $r'=1-|E({T_{appr}})|/n\in [r, r+2n^{-\alpha}]$ respectively. Pick $p$-random $V^{c,p}_{\mathrm{rand}}\subseteq V^c_{\mathrm{rand}}$ and $r$-random $C^{c,r}_{\mathrm{rand}}\subseteq C^c_{\mathrm{rand}}$.
If we have a matching approximation, then with high probability, Theorem~\ref{thm:maintheoremsemidisjoint} applies to $R_1=V^{c,p}_{\mathrm{rand}}, R_3=C^{c,r}_{\mathrm{rand}}, R_2=U_1^{\mathrm{rand}}\cup U_2^{\mathrm{rand}}, p=p, n=n$ (using Observation~\ref{Observation_matching-like} to establish all the conditions on $p$). In we have a path approximation, then with high probability, Lemma~\ref{lem:pathlikemain} applies to $V_{mid}=V^{c,p}_{\mathrm{rand}}, C=C^{c,r}_{\mathrm{rand}}, V_{str}=U_1^{\mathrm{rand}}, V_{end}=U_2^{\mathrm{rand}}, p=p,q=q, t=t, n=n$ (using Observation~\ref{Observation_path-like} to establish all the conditions on $p,q,t$).  With high probability $f$ is extendable to a pseudoembedding $h$ into $(V_{target}, C_{target})$ and $f(T_{appr})$ is a rainbow copy of $T_{appr}$ contained in $(V_{target},  C_{target})$. With high probability the sizes of all sets are within $n^{1-\gamma}$ of their expectations.
Fix an outcome for which all of these hold simultaneously. To prove the lemma it is sufficient to extend $f$ to a rainbow embedding $f'$ into $(V_{target}, C_{target})$ for such outcomes. 

We have that $|V^c_{\mathrm{rand}}\Delta V^{c,p}_{\mathrm{rand}}|\leq 3n^{1-\gamma}$ and  $|C^c_{\mathrm{rand}}\Delta C^{c,r}_{\mathrm{rand}}|\leq 3n^{1-\gamma}$.
Note that using the definition of $n^{-\mu}$-uniformity, and the set-theoretic identities $A\Delta B\subseteq (A\Delta C) \cup (C\Delta B)$ and $(A\setminus B) \Delta C^c\subseteq A^c \cup (B\Delta C)$, we have the following.  
\par \textbullet  $|f(U_1(T))\Delta U_1^{\mathrm{rand}}|, |f(U_2(T))\Delta U_2^{\mathrm{rand}}|$, $|f(U(T))\Delta U^{\mathrm{rand}}|\leq 2n^{1-\mu}\leq n^{1-\gamma}\leq p(T)^{10^{18}}n/3\log(n)^{10^{18}}$
\par \textbullet $|(V_{target}\setminus f(V(T_{appr}))\Delta V^{c,p}_{\mathrm{rand}}|
\leq |(V_{target}\setminus f(V(T_{appr}))\Delta V^{c}_{\mathrm{rand}}|+|V^c_{\mathrm{rand}}\Delta V^{c,p}_{\mathrm{rand}}|\leq$\\$ |V_{target}^c|+|V(f(T_{appr}))\Delta V_{\mathrm{rand}}|+3n^{1-\gamma}\leq n^{1-\mu}+3n^{1-\gamma}\leq 4n^{1-\gamma}\leq p(T)^{10^{18}}n/3\log(n)^{10^{18}}$.
\par \textbullet $|(C_{target}\setminus C(f(T_{appr})))\Delta C^{c,r}_{\mathrm{rand}}|\leq|(C_{target}\setminus C(f(T_{appr})))\Delta C^{c}_{\mathrm{rand}}|+|C^c_{\mathrm{rand}}\Delta C^{c,r}_{\mathrm{rand}}| \leq  |C_{target}^c|+|C(f(T_{appr}))\Delta C_{\mathrm{rand}}|+3n^{1-\gamma} \leq n^{1-\mu}+3n^{1-\gamma}\leq 4n^{1-\gamma}\leq p(T)^{10^{18}}n/3\log(n)^{10^{18}}$ \\
\par \textbf{Suppose $T$ is matching-like:}
Note that we have
\begin{align}
\sum &f(U(T))+\sum  V_{target}\setminus f(V(T_{appr})) -\sum C_{target}\setminus C(f(T_{appr})) \label{e1}\\
&= \sum h(U(T))+\sum  V_{target}\setminus h(V(T_{appr})) -\sum C_{target}\setminus C(h(T_{appr}))\label{e2}\\
&= \sum h(U(T))+\sum  V_{target} -\sum h(V(T_{appr}))-\sum C_{target}+\sum C(h(T_{appr}))\label{e3}\\
&= \sum h(U(T))+\sum  V_{target} -\sum h(V(T_{appr}))-\sum C_{target}+\sum_{v\in T_{appr}} h(v)d_{T_{appr}}(v)\label{e4}\\
&= \sum h(U(T))+\sum  V_{target} -\sum h(V(T_{appr}))-\sum C_{target}+\sum_{v\in T_{appr}} h(v)d_{T}(v)- \sum_{v\in U} h(v)\label{e5}\\
&=\sum  V_{target} -\sum h(V(T_{appr}))-\sum C_{target}+\sum_{v\in T_{appr}} h(v)d_{T}(v)\label{e6}\\
&=\sum  V_{target} -\sum h(V(T_{appr}))-\sum C_{target}+\sum_{v\in T} h(v)d_{T}(v)-\sum_{v\not\in T_{appr}} h(v)d_{T}(v)\label{e7}\\
&=\sum  V_{target} -\sum h(V(T_{appr}))-\sum C_{target}+\sum_{v\in T} h(v)d_{T}(v)-\sum_{v\not\in T_{appr}} h(v)\label{e8}\\
&=-\sum C_{target}+\sum_{v\in V(T)}d_T(v)h(v)
=0\label{e9}
\end{align}
Here (\ref{e2}) holds because $h$ agrees with $f$ on $T_{appr}$ and $U(T) \subseteq V(T_{appr})$, (\ref{e3}) holds because $h(V(T_{appr}))\subseteq V_{target}$, $h(C(T_{appr}))\subseteq C_{target}$, (\ref{e4}) holds because for any injection $h: V(T)\to V(K_G)$ we have $\sum_{xy\in E(T_{appr})} c(h(xy))=\sum_{x\in V(T_{appr})}h(x)d_{T_{appr}}(x)$, (\ref{e5}) holds because vertices in $U$ have one edge going outside $T_{appr}$ and vertices in $T_{appr}\setminus U$ have no such edges, (\ref{e6}) is just cancelling $\sum h(U(T))=\sum_{v\in U}h(v)$, (\ref{e7}) is splitting the sum $\sum_{v\in T_{appr}}h(v)d_T(v)$ into two, (\ref{e8}) is using that outside $T_{appr}$ we only have leaves, and (\ref{e9}) is using that $h$ is a pseudoembedding into $(V_{target}, C_{target})$.
By Theorem~\ref{thm:maintheoremsemidisjoint} (applied with $X=f(U(T)), Y=V_{target}\setminus f(V(T_{appr})), Z=C_{target}\setminus V(f(T_{appr}))$, $p=p(T), n=n, G=G,  R_1=U_1^{\mathrm{rand}}\cup U_2^{\mathrm{rand}}, R_2=V_{\mathrm{rand}}^c, R_3=C_{\mathrm{rand}}^c$), there is a rainbow matching from $f(U(T))$ to $V_{target}\setminus f(V(T_{appr}))$ using the colours $C_{target}\setminus C(f(T_{appr}))$. Adding this matching to the tree $f(T_{appr})$ gives a rainbow embedding of $T$.

\textbf{Suppose $T$ is path-like:} 
 Note that we have 
\begin{align}
&\sum f(U(T))+2\sum  V_{target}\setminus f(V(T_{appr})) -\sum C_{target}\setminus C(f(T_{appr}))\label{f1}\\
&= \sum h(U(T))+2\sum  V_{target}\setminus h(V(T_{appr})) -\sum C_{target}\setminus C(h(T_{appr}))\label{f2}\\
&= \sum h(U(T))+2\sum  V_{target} -2\sum h(V(T_{appr}))-\sum C_{target}+\sum C(h(T_{appr}))\label{f3}\\
&= \sum h(U(T))+2\sum  V_{target} -2\sum h(V(T_{appr}))-\sum C_{target}+\sum_{v\in T_{appr}} h(v)d_{T_{appr}}(v)\label{f4}\\
&= \sum h(U(T))+2\sum  V_{target} -2\sum h(V(T_{appr}))-\sum C_{target}+\sum_{v\in T_{appr}} h(v)d_{T}(v)- \sum_{v\in U} h(v)\label{f5}\\
&=2\sum  V_{target} -2\sum h(V(T_{appr}))-\sum C_{target}+\sum_{v\in T_{appr}} h(v)d_{T}(v)\label{f6}\\
&=2\sum  V_{target} -2\sum h(V(T_{appr}))-\sum C_{target}+\sum_{v\in T} h(v)d_{T}(v)-\sum_{v\not\in T_{appr}} h(v)d_{T}(v)\label{f7}\\
&=2\sum  V_{target} -2\sum h(V(T_{appr}))-\sum C_{target}+\sum_{v\in T} h(v)d_{T}(v)-2\sum_{v\not\in T_{appr}} h(v)\label{f8}\\
&=-\sum C_{target}+\sum_{v\in V(T)}d_T(v)h(v)
=0\label{f9}
\end{align}
Here (\ref{f2}) holds because $h$ agrees with $f$ on $T_{appr}$ and $U(T) \subseteq V(T_{appr})$, (\ref{f3}) holds because $h(V(T_{appr}))\subseteq V_{target}$, $h(C(T_{appr}))\subseteq C_{target}$, (\ref{f4}) holds because for any injection $h: V(T)\to V(K_G)$ we have $\sum_{xy\in E(T_{appr})} c(h(xy))=\sum_{x\in V(T_{appr})}h(x)d_{T_{appr}}(x)$, (\ref{f5}) holds because vertices in $U$ have one edge going outside $T_{appr}$ and vertices in $T_{appr}\setminus U$ have no such edges, (\ref{f6}) is just cancelling $\sum h(U(T))=\sum_{v\in U}h(v)$, (\ref{f7}) is splitting the sum $\sum_{v\in T_{appr}}h(v)d_T(v)$ into two, (\ref{f8}) is using that outside $T_{appr}$ we only have degree $2$ vertices, and (\ref{f9}) is using that $h$ is a pseudoembedding into $(V_{target}, C_{target})$.
 By Lemma~\ref{lem:pathlikemain} (applied with $V_{str}'=f(U_1(T)), V_{end}'=f(U_2(T)), V_{mid}'=V_{target}\setminus f(V(T_{appr})), C'=C_{target}\setminus V(f(T_{appr}))$, $p=p(T), t=t(T), q=q(T), n=n,$ $G=G,  V_{str}=U_1^{\mathrm{rand}},  V_{end}=U_2^{\mathrm{rand}}, V_{mid}=V_{\mathrm{rand}}^c, C=C_{\mathrm{rand}}^c$), there is a system of rainbow paths $P$ of length $t$ connecting corresponding vertices in $U_1/U_2$ with $C(P)=C_{target}\setminus C(f(T_{appr}))$ and $V(P)=V_{target}\setminus f(V(T_{appr}))\cup U$. Adding these paths to the tree $f(T_{appr})$ gives a rainbow embedding of $T$.
 \end{proof}

To apply the above lemma, we need to construct pseudoembeddings. Since this amounts to constructing injections with specific sums, we now develop machinery for finding elements in a group with prescribed sum.
\begin{lemma}\label{Lemma_three_elements_prescribed_sum_random}
Let $\Delta^{-1} \gg \mu \gg \rho\gg n^{-1}$.
Let $X$ be a $(\geq n^{-\rho})$-random subset of an abelian group $G$. With high probability the following holds. For every $b\in G$, $N,U\subseteq G$ with $|U|\leq n^{1-\mu}, |N|\leq \Delta$, there are elements $x,y,z\in X\setminus U$ with $x+y+z=b$ in $G$ and $x+N, y+N, z+N$ pairwise disjoint and contained in $X\setminus U$.
\end{lemma}
\begin{proof}
Without loss of generality, we may assume that $X$ is $p:=n^{-\rho}$-random (by passing to a subset of this probability).
First, fix some $b$ and $N\subseteq G$ as in the lemma.
Call a triple $(x,y,z)$ \textit{good} if $x+y+z=b$, and $x+N, y+N, z+N$ are disjoint.
There are precisely $n^2$ solutions to $x+y+z=b$. 
For every $w\in G$, there are $n$ solutions to $x+y+z=b$ having $x-y=w$ (or $y-z=w$, or $z-x=w$). Thus there are at most $6|N|^2n$ solutions to $x+y+z=b$ with $\{x-y,y-x,y-z,z-y,z-x,x-z\}\cap (N-N) \neq \emptyset$.  This leaves $\geq n^2-6|N|^2 n\geq n^2/2$ solutions for which this doesn't happen i.e. there are $\geq n^2/2$ good triples in $G$.

Given a good triple $t=(x,y,z)$, let $S(t):=\{x,y,z\}\cup (x+N)\cup (y+N)\cup (z+N)$. Note that we always have $|S(t)|\leq 3(|N|+1)$.
Let $t_1, \dots, t_m$ be a maximal collection of good triples which have all the sets $S(t_1), \dots, S(t_m)$ pairwise disjoint. 
Letting $S:=S(t_1)\cup\dots\cup S(t_m)$, we have $|S|\leq 3m(|N|+1)$ and for all good triples $t$, $S(t)$ intersects $S$. For every $s\in S$ and $w\in N\cup\{0\}$, there are $n$ solutions to $x+y+z=b$ having $x+w=s$ (or $y+w=s$, or $z+w=s$), giving at most $3|S|(|N|+1)n$ solutions with $S(\{x,y,z\})\cap S\neq \emptyset$. This shows that there are at most $3|S|(|N|+1)n$ good triples, which implies $3|S|(|N|+1)n\geq n^2/2$, and hence $m\geq \frac{|S|}{3|N|+3}\geq \frac{n}{18(|N|+1)^2} \geq n/36\Delta^2$.

By linearity of expectation, the expected number of these $m$ triples with $x+N,y+N, z+N\subseteq X$ is $\geq  p^{6\Delta}m\geq p^{6\Delta}n/36\Delta^2$. Using disjointness, each of the $m$ triples has ``$x+N,y+N, z+N\subseteq X$'' independently, and so by Chernoff's bound we have that with  probability $\geq1-o(n^{-\Delta-1})$, there are $> p^{6\Delta}n/40\Delta^2$ disjoint good triples with $x+N,y+N, z+N\subseteq X$. Since $|U|\leq n^{-\mu}< n^{-6\rho\Delta}/40\Delta^2= p^{6\Delta}n/40\Delta^2$, at least one of these has $x+N,y+N, z+N$ disjoint from any given $U$ and it satisfies the lemma. Taking a union bound over all $N, b$ proves the result.
\end{proof}

The following is an easier to use version of the above lemma.
\begin{lemma}\label{Lemma_neighbourhood_prescribed_sum_random}
Let $\Delta^{-1} \gg \mu \gg \rho\gg n^{-1}$.
Let $C,V$ be $(\geq n^{-\rho})$-random independent subsets of an abelian group $G$. With high probability we have the following. 
\begin{enumerate}[(E1)]
\item Let $U,N\subseteq V(K_G)$ with $|U|\leq n^{1-\mu}, |N|\leq \Delta$. For any $b\in G$, there are distinct elements $x,y,z\in V\setminus U$ with $x+y+z=b$ with $K_G[\{x,y,z\}, N]$ being rainbow with all colours contained in $C\setminus U$
\item Let $U,N\subseteq V(K_G)$ with $|U|\leq n^{1-\mu}, |N|\leq \Delta$. There is some $x\in V\setminus U$ with all edges $yx$ for $y\in N$ having colour in $C\setminus U$.\end{enumerate}
\end{lemma}
\begin{proof}
Let $X=C\cap V$ to get a $p^2$-random set. Lemma~\ref{Lemma_three_elements_prescribed_sum_random} applies to $X$, i.e., we have that for any $b\in G$, $N,U\subseteq G$ with $|U|\leq p^{6\Delta} n/40\Delta, |N|\leq \Delta$, there are elements $x,y,z\in X\setminus U\subseteq V\setminus U$ with $x+y+z=b$ in $G$ and $x+N, y+N, z+N$ pairwise disjoint and contained in $X\setminus U$. But ``$x+N, y+N, z+N$ pairwise disjoint and contained in $X\setminus U$'' implies that $x,y,z\subseteq X\setminus U\subseteq V\setminus U$ and that $K_G[\{x,y,z\}, N]$ is rainbow with all colours in $X\setminus U\subseteq C\setminus U$ --- which is what (E1) asks for. 

For (E2), apply part (E1) with any $b$ and note that the resulting $x$ satisfies (E2).
\end{proof}

The following technical lemma is later used to extend a rainbow embedding of an approximation of a tree $T$ into a pseudoembedding of $T$ (with a view of then combining this with Lemma~\ref{Lemma_final_extension} to get a rainbow embedding of $T$).
\begin{lemma}\label{Lemma_modify_random_embedding}
Let $\Delta^{-1} \gg \alpha\gg \mu \gg \rho\gg n^{-1}$.
Let $T$ be a forest with $\Delta(T)\leq \Delta$ and $|V(T)|\leq n-n^{1-\rho}$.
Let $V_{target}\subseteq V(K_G)$, $C_{target}\subseteq C(K_G)$ with $|V_{target}|, |C_{target}|\geq n-n^{1-\alpha}$. Let $T_{core}\subseteq T$ be an induced subforest of size $\leq n^{1-\alpha}$. For $k\leq \Delta$, let $D_1, \dots, D_{k}\subseteq V(T)$ be disjoint subsets with $V(T)\setminus \bigcup_{i=1}^kD_i\subseteq V(T_{core})$ and $|D_i\setminus V(T_{core})|\geq 6$ for all $i$.

Let $h:V(T)\to V(K_G)$ be an injection which is a rainbow embedding into $(V_{target}, C_{target})$, when restricted to $T_{core}$  and $f:V(T)\to V(K_G)$ a  rainbow embedding. Let $V_1\subseteq V(K_G)$, $C_1\subseteq C(K_G)$ satisfy (E1)  and (E2) and  $U_f:=(V_1\cap V(f(V(T)))\cup (C_1\cap C(f(V(T)))$  has $|U_f|\leq n^{1-\alpha}$.

Then there is a rainbow embedding  $f':V(T)\to V(K_G)$ into $(V_{target}, C_{target})$ agreeing with $h$ on $T_{core}$,  disagreeing with $f$ on $\leq n^{1-\mu}$ vertices and with $\sum f'(D_i)= \sum h(D_i)$ for all $i$.
\end{lemma}
\begin{proof}
Let $A_1= V(T_{core})$.  Let $A_2\subseteq V(T)\setminus V(T_{core})$ consist of an independent set of size $3$ in each $D_i$ (which exists because each $|D_i\cap (V(T)\setminus V(T_{core})|\geq 6$ and $T$ is bipartite). 
Let $A_3^V$ be the set of vertices $v\in V(T)$ with $f(v)\in (V_{target}^c \cup V(h(T_{core}))$. 
Let $A_3^C$ be the set of vertices $v\in V(T)$ contained in edges $vu\in E(T)$ with $c(f(vu))\in (C_{target}^c \cup C(h(T_{core}))$.  
Let $A_3=(N_T(A_1\cup A_2)\cup A_3^V\cup A_3^C)\setminus (A_1\cup A_2)$.
Let  $A_4=V(T)\setminus (A_1\cup A_2\cup A_3)$.
Define $f_1:A_1\cup A_4\to V(K_G)$ to agree with $h$ on $A_1$ and agree with $f$ on $A_4$. Note that $f_1$ is a rainbow embedding of $T[A_1\cup A_4]$ into $(V_{target}, C_{target})$ since $h$ is a rainbow embedding of $T[A_1]=T_{core}$, $f$ is a rainbow embedding of $T[A_4]$ into $(V_{target}, C_{target})$ (using that $A_4$ is vertex-disjoint from $A_3^V$ and $A_3^C$), $h(T[A_1])$ and $f(T[A_4])$ are vertex-disjoint and colour-disjoint  (using that $A_4$ is vertex-disjoint from $A_3^V$ and $A_3^C$), and there are no edges between $A_1$ and $A_4$ since all edges of $T$ from $A_1$ go to $A_1\cup A_2\cup A_3$. 
Let $U_{target}$ be the vertices/colours of $V_1\setminus V_{target}$ and $C_1\setminus C_{target}$, noting that $|U_{target}|\leq |V_{target}^c|+|C_{target}^c|\leq 2n^{1-\alpha}.$
Let $U_{f_1}$ be the vertices/colours of  $f_1(T[A_1\cup A_4])$ which are in  $C_1\cup V_1$, noting that $|U_{f_1}|\leq |U_f|+|E(T_{core})|+|V(T_{core})|\leq 3n^{1-\alpha}$.
 
 \begin{claim}\label{Claim_modify_random_embedding1}
 We can  extend $f_1$ to an embedding $f_2: A_1\cup A_3\cup A_4\to im_{f_1}\cup (V_1\cap V_{target})$ with additional  colours used in $C_1\cap C_{target}$
 \end{claim}
 \begin{proof}
Order $A_3=\{a_1, \dots, a_t\}$, noting that $t\leq |N_T(A_1)|+ |N_T(A_2)|+|A_3^V|+ |A_3^C|\leq \Delta|V(T_{core})|+3\Delta+(|V_{target}^c|+|V(T_{core})|)+2(|C_{target}^c|+|E(T_{core})|)\leq 6\Delta n^{1-\alpha}$.
Define $T_i=T[A_1\cup A_4\cup\{a_1, \dots, a_i\}]$.
Setting $g_0=f_1$, we build functions $g_1, \dots, g_t$ with $g_i:V(T_i)\to V(K_G)$ being a rainbow embedding of $T_i$ into $(V_{target}, C_{target})$ extending $g_{i-1}$.
To construct $g_i$, set $N_i=g_{i-1}(N_{T_i}(a_i))$, $U_i=U_{target}\cup  V(g_{i-1}(T_{i-1}))\cup C(g_{i-1}(T_{i-1}))$ noting $|N_i|\leq \Delta$ and $|U_i\cap (C_1\cup V_1)|\leq |U_{target}|+|U_{f_1}| + (\Delta+1)(i-1)\leq 14\Delta^2 n^{1-\alpha}\leq n^{1-\mu}$.  Apply (E2) to get a vertex $x_i\in V_1\setminus U$ with all edges from $x$ to $N_i$ having colours in $C_1\setminus U$. Defining $g_i(a_i)=x_i$ we get a rainbow embedding of $T_i$ into $(V_{target}, C_{target})$.
 \end{proof}
 
Let $U_{f_2}$ be the vertices/colours of  $f_2(T[A_1\cup A_3\cup A_4])$ which are in  $C_1\cup V_1$, noting that $|U_{f_2}|\leq |U_{f_1}|+  (\Delta+1)|A_3|\leq 14\Delta^2 n^{1-\alpha}$.

 \begin{claim}\label{Claim_modify_random_embedding2}
 We can  extend $f_2$ to an embedding $f': A_1\cup A_2\cup A_3\cup A_4\to im_{f_2}\cup (V_1\cap V_{target})$ using colours of $C_1\cap C_{target}$ such that for all $i$  we have $\sum f'(D_i)=\sum h(D_i)$.
 \end{claim}
 \begin{proof}
Let $A_2=\{a_1,b_1, c_1,  \dots, a_t, b_t, c_t\}$ where for each $i$, $\{a_i, b_i, c_i\}\subseteq D_i$ is an independent set of size 3. Note $|A_2|\leq 3\Delta$. For each $i\leq t$, let $T_i=T[A_1\cup A_3\cup A_4\cup\{a_1,b_1, c_1, \dots, a_i, b_i, c_i\}]$ and  $\sigma_i=\sum h(D_i)-\sum f_2(D_i\setminus \{a_i, b_i, c_i\})$.
Setting $g_0=f_2$, we build functions $g_1, \dots, g_t$ with $g_i:V(T_i)\to V(K_G)$ being a rainbow embedding of $T_i$ into $(V_{target}, C_{target})$ extending $g_{i-1}$.

To construct $g_i$, set $N_i=g_{i-1}(N_{T_i}(\{a_i, b_i, c_i\}))$, $U_i=U_{target}\cup V(g_{i-1}(T_{i-1}))\cup C(g_{i-1}(T_{i-1}))$ noting $|N_i|\leq 3\Delta$ and $|U_i\cap (C_1\cup V_1)|\leq |U_{f_2}| + (3\Delta+3)(i-1)\leq 40\Delta^2 n^{1-\alpha}\leq n^{1-\mu}$.  Apply (E1) to get distinct vertices $x_i, y_i, z_i\in V_1\setminus U_i$ with all edges from $\{x_i, y_i, z_i\}$ to $N_i$ having different colours in $C_1\setminus U_i$ and $x_i+y_i+z_i=\sigma_i$. Defining $g_i(a_i)=x_i, g_i(b_i)=y_i, g_i(c_i)=z_i$ we get a rainbow embedding of $T_i$ into $(V_{target}, C_{target})$. 

Set $f'=g_t$. We have  $\sum f'(D_i)=\sum g_0(D_{i}\setminus\{a_i, b_i, c_i\})+x_i+y_i+z_i=\sum f_2(D_{i}\setminus\{a_i, b_i, c_i\})+\sigma_i=\sum h(D_i)$.
\end{proof}
By construction, we have that $f'$ agrees with $h$ on $T_{core}=A_1$, $\sum f'(D_i)= \sum h(D_i)$ for all $i$, and $f'$ disagrees with $f$ on the subset $A_1\cup A_2\cup A_3$ which has order $\leq |V({T_{core}})|+ 3\Delta+ 6\Delta n^{1-\alpha}\le n^{1-\mu}$.
\end{proof}

We'll need the following consequence of Lemma~\ref{Lemma_three_elements_prescribed_sum_random}.
\begin{lemma}\label{Lemma_partition_zero_sum}
Let $\Delta^{-1}\gg \alpha\gg n^{-1}$.
Let $S\subseteq G$ have $|S|\geq n-n^{1-\alpha}$ and $\sum  S=0$. For any $m_1, \dots, m_{\Delta}\geq 3$ with $\sum m_i=|S|$, $S$ can be partitioned into zero-sum sets of orders $m_1, \dots, m_{\Delta}$.
\end{lemma}
\begin{proof}
Let $\Delta\gg  \alpha\gg\mu\gg \rho\gg n^{-1}$.
For some $i$, we have $m_i\geq |S|/\Delta\geq (n-n^{1-\alpha})/\Delta\geq n/2\Delta\geq n^{1-\rho}$. Let $X$ be an $(m_i/n +3n^{-\alpha})$-random subset of $V(K_G)$, noting this probability is $\geq n^{-\rho}$. With high probability, $|X| \in [m_i+2 n^{1-\alpha}, m_i+4n^{1-\alpha}]$, which combined with $|S|\geq n-n^{1-\alpha}$ gives $|X\cap S|\in [m_i+n^{1-\alpha}, m_i+4n^{1-\alpha}]\subseteq [m_i+3\Delta-3, m_i+4n^{1-\alpha}]$. Also with high probability, $X$ satisfies the property of Lemma~\ref{Lemma_three_elements_prescribed_sum_random}.   

Let $X'\subseteq X\cap S$ be a subset of size exactly $m_i+3\Delta-3$. Partition $S\setminus X'$ arbitrarily into sets $M_1, \dots, M_{i-1}, M_{i+1}, \dots, M_{\Delta}$ with $|M_j|=m_j-3$. For $j=1, \dots, i-1, i+1, \dots, \Delta$, use Lemma~\ref{Lemma_three_elements_prescribed_sum_random} with $X=X$, $N=\emptyset$   to find disjoint triples $\{x_j, y_j, z_j\}\subseteq X'$ with $x_j+y_j+z_j=-\sum M_j$   (at the $j$th application set $U=\{x_t, y_t, z_t: t<j\}\cup (X\setminus X')$ which has order $|U|\leq 3\Delta+5n^{1-\alpha}\leq n^{1-\mu}$). Now set $M_j'=\{x_j, y_j, z_j\}\cup M_j$ for $j\neq i$, and $M_i'=X'\setminus \{x_j, y_j, z_j: j\in [\Delta]\setminus i\}$. We have $\sum M_j'=0$ for $j\neq i$ by choice of $x_j, y_j, z_j$, and $\sum M_i'=\sum S-\sum_{j\neq i} \sum M_j'=0-0=0$. Thus the sets $M_1', \dots, M_{\Delta}'$ give the partition we want. 
\end{proof}
 
The following lemma transforms an embedding of a core of a tree $T$ into a pseudoembedding of $T$.
\begin{lemma}\label{Lemma_pseudoembed_tree_with_prescribed_core}
Let $\Delta^{-1}\gg \alpha\gg n^{-1}$.
Let $G$ be a group and $T$ a bounded degree tree with $\Delta(T)=\Delta$. Let $V_{target},C_{target}\subseteq G$ with $|T|=|V_{target}|=|C_{target}|+1\geq (1-n^{-\alpha})n$. In the case $G=\mathbb{Z}_2^m$, assume $\id \not\in C_{target}$. Let $T_{core}$ be a core of $T$ of size $\leq n^{1-\alpha}$. Then any rainbow embedding $\phi$  of $T_{core}$ into $(V_{target}, C_{target})$ with $\sum_{v\in V(T_{core})}d_T(v)\phi(v)=\sum C_{target}$ and $\sum \phi(V(T_{core}))=\sum V_{target}$ extends to a pseudoembedding $h$ of $T$ into $(V_{target}, C_{target})$.
\end{lemma}
\begin{proof}
For each $d=1, \dots, \Delta$, let $m_d$ be the number of vertices of degree $d$ that $T$ has outside $T_{core}$, noting these  are either $=0$ or  $\geq 6$ (by definition of $T_{core}$).
Note that we're assuming $\sum V_{target}\setminus \phi(T_{core})=0$, so we can use Lemma~\ref{Lemma_partition_zero_sum} (with $\alpha'=\alpha/2$) to partition $V_{target}\setminus \phi(T_{core})$ into zero-sum sets $M_1, \dots, M_{\Delta}$ of sizes $m_1, \dots, m_d$ respectively. Now extend $\phi$ into $h$ by embedding the degree $d$ vertices outside $\phi(T_{core})$ to $M_d$ arbitrarily. This ensures that $\sum_{v\in V(T)}d_T(v)h(v)=\sum_{v\in V(T_{core})}d_T(v)\phi(v)=\sum C_{target}$. We also  constructed $h$ so that $h(T)=\phi(T_{core})\cup M_1\cup \dots \cup M_{\Delta}=V_{target}$ --- thus $h$ is a pseudoembedding.
\end{proof}

The following lemma is exactly the same as the previous one, except that it produces a rainbow embedding of $T$, rather than just a pseudoembedding.
\begin{lemma}\label{Lemma_embed_tree_prescribed_core}
Let $\Delta^{-1}\gg \alpha\gg n^{-1}$.
Let $G$ be a group and $T$ a bounded degree tree with $\Delta(T)=\Delta$. Let $V,C\subseteq G$ with $|T|=|V_{target}|=|C_{target}|+1\geq (1-n^{-\alpha})n$. In the case $G=\mathbb{Z}_2^m$, assume $\id \not\in C$. Let $T_{core}$ be a core of $T$ of size $\leq n^{1-\alpha}$. Then any rainbow embedding $\phi$  of $T_{core}$ into $(V_{target}, C_{target})$ with $\sum_{v\in V(T_{core})}d_T(v)\phi(v)=\sum C_{target}$ and $\sum \phi(V(T_{core}))=\sum V_{target}$ extends to a rainbow embedding $f$ of $T$ into $(V_{target}, C_{target})$.
\end{lemma} 
\begin{proof}
Let $\Delta^{-1}\gg\alpha\gg\mu\gg \rho\gg n^{-1}$.
Use Lemma~\ref{Lemma_pseudoembed_tree_with_prescribed_core} to extend $\phi$ to a pseudoembedding $h$ of $T$ into  $(V_{target}, C_{target})$. Use Lemma~\ref{Lemma_nonexhausted_vertices_in_Tabs} to find an approximation $T_{appr}$ of $T$ containing all the vertices of $T_{core}$.
For each non-exhausted degree $d$, let $D_d=\{v\in V(T_{appr}): d_T(v)=d\}$ noting that $V(T_{appr})\setminus \bigcup D_d\subseteq V(T_{core})$ (since all vertices of exhausted degrees are in $T_{core}$) and $|D_d\setminus V(T_{core})|\geq 6$ for all $d$ (from Lemma~\ref{Lemma_nonexhausted_vertices_in_Tabs}).
\begin{claim}
There is a random $f':V(T_{appr})\to V(K_G)$ which is $n^{-\mu/2}$-uniform on $$(V(T_{appr}), E(T_{appr}), U_1(T), U_2(T))$$ and with high probability has:
\begin{enumerate}[(i)]
\item $f'$ agrees with $h$ on $T_{core}$.
\item $f'$ is a rainbow embedding of $T_{appr}$ into $(V_{target}, C_{target})$.
\item $\sum h(D_d)=\sum f'(D_d)$ for each non-exhausted degree $d$.
\end{enumerate}
\end{claim}
\begin{proof}

Recalling that $|T_{appr}|\leq (1-n^{-\rho})n$ (from Observations~\ref{Observation_matching-like},~\ref{Observation_path-like} (ii), (iv)), apply Lemma~\ref{Lemma_approximate_embedding} to get a random   $f: T_{appr}\to K_G$ which is $n^{-\alpha}$-uniform on $\{V(T), E(T), U_1(T), U_2(T)\}$ and is a rainbow embedding of $T$ with high probability. Let $V_0, C_0$ witness the $n^{-\alpha}$-uniformity of $f$ on $\{V(T), E(T)\}$, noting that these are $\leq (1-n^{-\rho})$-random and independent.
Let $V_1:=V(K_G)\setminus V_0$ and $C_1:=C(K_G)\setminus C_0$, to get $\geq n^{-\rho}$-random, independent sets for which $U_f:=(V_1\cap V(f(V(T)))\cup (C_1\cap C(f(V(T)))$  has $|U_f|\leq 2n^{1-\alpha}$ with high probability. By Lemma~\ref{Lemma_neighbourhood_prescribed_sum_random}, $V_1, C_1$ satisfy (E1) and (E2) with high probability.
Call an outcome good if  these the above happen, noting that we have a good outcome with high probability.

Define $f'$ as follows: for each good outcome, apply Lemma~\ref{Lemma_modify_random_embedding} to $T=T_{appr}, T_{core}, V_{target}, C_{target}$, $f, h, V_1, C_1$ in order to get a rainbow embedding   $f':V(T_{appr})\to V(K_G)$. For bad outcomes, define $f'$ arbitrarily.
Note that for good outcomes, Lemma~\ref{Lemma_modify_random_embedding} guarantees (i) -- (iii), and hence these hold with high probability. Also,  $f'$ is still $n^{-\mu/2}$-uniform on $(V(T), E(T), U_1(T), U_2(T))$ (since $f'$ and $f$ differ on at most $n^{1-\mu}$ vertices).
\end{proof}

\begin{claim}\label{Claim_embed_tree_prescribed_core}
With high probability, $f'$ extends to a pseudoembedding $h':V(T)\to V(K_G)$ into $(V_{target}, C_{target})$. 
\end{claim}
\begin{proof}
Consider any outcome for which $f'$ satisfies (i) -- (iii). Extend $f'$ to an arbitrary bijection $h':V(T)\to V_{target}$ arbitrarily.
Since the outcomes we are considering occur with high probability, it is sufficient to prove that $h'$  is a pseudoembedding into $(V_{target}, C_{target})$ i.e. to prove that  $\sum_{v\in T}d_T(v)h'(v)=\sum C_{target}$.

Defining  $D_d=\{v\in V(T_{appr}): d_T(v)=d\}$ for exhausted as well as non-exhausted degrees, note that we actually have ``$\sum f'(D_d)=\sum h(D_d)$''  for \emph{all} $d$ --- for non-exhausted degrees this is (iii), while for exhausted degrees this happens because $f'$ agrees with $h$ on $T_{core}$ and $D_d\subseteq V(T_{core})$ by definition of ``core''. This implies that $\sum_{v\in V(T_{appr})}d_T(v)h'(v)=\sum_{v\in V(T_{appr})}d_T(v)f'(v)=\sum_{d=1}^{\Delta}d\sum f'(D_d)=\sum_{d=1}^{\Delta}d\sum h(D_d) =\sum_{v\in V(T_{appr})}d_T(v)h(v)$ and $\sum f'(V(T_{appr}))=\sum_{d=1}^{\Delta}\sum f'(D_d)=\sum_{d=1}^{\Delta}\sum h(D_d)=\sum h(V(T_{appr}))$.
Then $\sum h'(V(T_{appr}))=\sum f'(V(T_{appr}))=\sum h(V(T_{appr}))$ together with the fact that $h',h$ are both bijections from $T$ to $V_{target}$, shows that  $\sum h'(V(T)\setminus V(T_{appr}))=\sum h(V(T)\setminus V(T_{appr}))$. Since vertices outside $T_{appr}$ all have the same degree (either $1$ or $2$, depending on whether $T_{appr}$ is a matching-approximation or path-approximation), this implies that  $\sum_{v\not\in T_{appr}}d_T(v)h'(v)=\sum_{v\not\in T_{appr}}d_T(v)h(v)$. Adding this to our earlier equation ``$\sum_{v\in V(T_{appr})}d_T(v)h'(v)=\sum_{v\in V(T_{appr})}d_T(v)h(v)$'', and using that $h$ is a pseudoembedding into $(V_{target},C_{target})$ gives $\sum_{v\in V(T)}d_T(v)h'(v)=\sum_{v\in V(T)}d_T(v)h(v)=\sum C_{target}$ as required.
\end{proof}
By Lemma~\ref{Lemma_final_extension} with $f=f'$, there is a random rainbow embedding of $T$ into $(V_{target},C_{target})$. 
\end{proof}

We can now prove the main result of this section, Theorem~\ref{Theorem_embed_tree_prescribed_sets}, restated below for convenience.

\vspace{2mm}
\noindent \textbf{Theorem~\ref{Theorem_embed_tree_prescribed_sets}.}
Let $\Delta^{-1}\gg\mu\gg n^{-1}$.
Let $G$ be a group and $T$ a bounded degree tree with $\Delta(T)=\Delta$. Let $V_{target},C_{target}\subseteq G$ with $|T|=|V_{target}|=|C_{target}|+1\geq (1-n^{-\mu})n$. In the case $G=\mathbb{Z}_2^m$, assume $\id \not\in C$. Let $T_{core}$ be a core of $T$ of size $\leq n^{1-\mu}$. Then,  (i) $T$ has a rainbow  embedding $f$ into $(V_{target},C_{target})$ if and only if (ii) there is a rainbow embedding $\phi$ of $T_{core}$ into $(V_{target},C_{target})$ with $\sum_{v\in V(T_{core})}d_T(v)\phi(v)=\sum C_{target}$ and $\sum \phi(V(T_{core}))=\sum V_{target}$.
\begin{proof}
(ii) $\implies$ (i): 
Lemma~\ref{Lemma_embed_tree_prescribed_core} gives a rainbow embedding into $(V_{target},C_{target})$ which extends $\phi$.

(i) $\implies$ (ii): We remark that the interest of this direction is theoretical, and this implication is not used in the remainder of the paper. 

\par Pick $1\gg \mu\gg \rho\gg n^{-1}$. Note that for each non-exhausted degree $d$ we can pick  an independent set $\{a_d, b_d, c_d\}$ of 3 degree $d$ vertices inside $T_{core}$ (since there are $\geq 6$ degree $d$ vertices there, and $T$ is bipartite). 
Let $T_{core}'$ be $T_{core}$ with $a_d, b_d, c_d$ deleted of each non-exhausted  degree $d$.  Let $\sigma_d=\sum\{f(v): d_T(v)=d, v\in T\setminus T_{core}'\}$ and $N_d=f(N_T(a_d)\cup N_T(b_d)\cup N_T(c_d))$. Use Lemma~\ref{Lemma_three_elements_prescribed_sum_random} with $X=V(K_G)$ (which is $(\geq n^{-\rho})$-random), $N=N_d$ to pick disjoint triples of distinct vertices $x_d, y_d, z_d$  with $x_d+y_d+z_d=\sigma_d$ and $K_G[\{x_d, y_d, z_d\}, N_d]$ rainbow  and disjoint from vertices/colours in  $V(f(T'_{core}))\cup C(f(T'_{core}))\cup (V(K_G)\setminus V_{target})\cup (C(K_G)\setminus C_{target})$ (which have total size $4n^{1-\mu}\leq n^{1-\mu/2}$). Construct $\phi$ to agree with $f$ on $T_{core}'$ and embed $(a_d, b_d, c_d)$ to $(x_d, y_d, z_d)$ for all non-exhausted degrees (in the below we abbreviate this as n.-e.). Then

\begin{align*}
\sum_{v\in V(T_{core})}d_T(v)\phi(v)
&=\sum_{v\in V(T_{core'})}d_T(v)f(v)+\sum_{d \text{ is n.-e.}}d(\phi(a_d)+\phi(b_d)+\phi(c_d)) \\
&=
\sum_{v\in V(T_{core'})}d_T(v)f(v)+\sum_{d \text{ is n.-e.}}d\sigma_d \\
&=\sum_{v\in V(T_{core'})}d_T(v)f(v)+
\sum_{d \text{is n.-e.}}\quad \sum_{v\in T\setminus T_{core'}, d_T(v)=d} df(v) \\
&=\sum_{v\in V(T)}d_T(v)f(v)
=\sum_{xy\in E(T)}(f(x)+f(y))
=\sum C_{target}
\end{align*}
Here the first equation uses the definition of $T_{core}'$ and the fact that $\phi(v)=f(v)$ outside $T_{core}'$. The second equation uses that we embedded $a_d, b_d, c_d$ to $x_d, y_d, z_d$ which sum to $\sigma_d$. The third equation uses the definition of $\sigma_d$. The fourth equation uses that all vertices of exhausted degrees are in $T_{core}'$.
The fifth equation uses that in the sum $\sum_{xy\in E(T)}(f(x)+f(y))$ every $f(v)$ occurs exactly $d_T(v)$ times. The sixth equation uses that $f$ is a rainbow embedding into $(V_{target}, C_{target})$ and $|C_{target}|= |T|-1=e(T)$, and so $f$ uses every colour of $C_{target}$ precisely once. Similar reasoning gives 
\begin{align*}
\sum\phi(V(T_{core}))
&=\sum f(V(T_{core}'))+\sum_{d \text{ is n.-e.}}(\phi(a_d)+\phi(b_d)+\phi(c_d)) \\
&=
\sum f(V(T_{core}'))+\sum_{d \text{ is n.-e.}}\sigma_d \\
&=\sum f(V(T_{core}'))+
\sum_{d \text{ is n.-e.}}\quad\sum_{v\in T\setminus T_{core'} \text{ and } d_T(v)=d} f(v) \\
&=\sum f(V(T_{core}))=\sum V_{target}.\end{align*}This concludes the proof.\end{proof}

\section{Characterizing harmonious trees}\label{sec:characterisation}
The goal of this section is to prove Theorem~\ref{Theorem_main_intro}. The proof uses Theorem~\ref{Theorem_embed_tree_prescribed_sets} to find the embedding --- and hence what we are really trying to understand is when a core of $T$ has a rainbow embedding into $K_G$ satisfying  (ii) of that theorem. In the next few pages we develop machinery for this.

 We call a multiset $\{g_1, \dots, g_k\}$ \textit{simple} if everything occurs with multiplicity $\leq 1$. Similarly, we call a sequence $(x_1, \dots, x_k)$ simple if all its terms are distinct.
Let $d=(d_1,\dots, d_k)\in \mathbb{Z}^k$ and $x=(x_1, \dots, x_k)\in G^k$ for an abelian group $G$, define the multiset $x\ast d:= \{x_i+x_j: 1\leq i< j\leq k\}\cup \{d_1x_1+\dots+d_kx_k\}$   (noting that this has $\binom k2+1$ elements). Given a set of vertices $v=(v_1, \dots, v_k)$ in a graph $T$,   define the multiset $x\ast_T v=\{x_i+x_j: v_iv_j\in E(T)\}\cup \{(-d_T(v_1)+1)x_1+\dots+(-d_T(v_k)+1)x_k\}$ (noting that this has $e(T[V])+1$ elements). Note that if we let $d=(-d_T(v_1)+1, \dots, -d_T(v_k)+1)$ then we have the multiset containment $x\ast_T v\subseteq x\ast d$ --- and therefore $x\ast d$ being simple implies that $x\ast_T v$ is simple.

 Recall that by the Fundamental Theorem of Abelian Groups, every abelian group is a direct product  $G\cong C_{m_1}\times C_{m_2}\times \dots \times C_{m_k}$, where each $m_i$ is a prime power. Given such an expression of $G$, for an element $g\in G$, we define the support of $g$ to be the coordinates $C_{m_i}$ on which $g\neq 0$. 
Note that given two elements $g,h$ with distinct supports, we have $g\neq h$. Thus, a convenient way of showing that some multiset $M\subseteq G$ over $G$ is simple is to show that all its elements have distinct supports.

For $d\in \mathbb{Z}$, and an abelian group $G$, let $f_{d,G}:G\to G$ be defined by $f_{d,G}:x\to dx$. 
\begin{observation}\label{Observation_mod_G_charterization} 
Let $G\cong C_{m_1}\times C_{m_2}\times \dots \times C_{m_k}$, where each $m_i$ is a prime power. For any $d\in \mathbb{Z}$, the following are equivalent.
\begin{enumerate}[(i)]
\item $d\equiv 0 \pmod{m_i}$ for $i=1, \dots, k$.
\item The function $f_{d,G}: G\to G$ with $f_{d,G}(x)=dx$ is identically $0$.
\end{enumerate}
\end{observation}
\begin{proof}
(i) $\implies$ (ii): For all $i$, we have $d\equiv 0 \pmod{m_i}$ which implies $dx= 0$ for all  $x\in C_{m_i}$. This, in turn, implies that $dx=0$ in $G$ and hence  $f_{d,G}(x)=dx$ is identically $0$.

(ii) $\implies$ (i): Let $x\in G$ be the element which equals $1$ on every coordinate. Then the fact that the $i$th  coordinate of $f_{d,G}(x)=dx$ equals zero implies that $d\equiv 0 \pmod{m_i}$. 
\end{proof}
From now on, we write $d\equiv d'\pmod G$ if either of the properties (i) or (ii) in the above observation hold for $d-d'$ (i.e. if $f_{d-d',G}$ is identically zero in $G$, or, equivalently if $d\equiv d'\pmod {m_i}$ for each $m_i$).

\begin{lemma}\label{Lemma_xastd_simple_cyclic} 
Let $k\geq 1$, $\Delta^{-1}, k^{-1}\gg n^{-1}$, $(d_1,\dots, d_k)\in [-\Delta,-1]^k$. There exists some simple $x\in (\mathbb{Z}_n)^k$ with $x\ast d$ simple.
\end{lemma}
\begin{proof}
$x_1:=2, x_2=4, \dots, x_k=2^{k}$. Using  $k^{-1}\gg n^{-1}$, we have that $x_1, \dots, x_{k-1}$ are distinct modulo $n$ as is everything in $\{x_i+x_j: i< j\}$. Also everything in $\{x_i+x_j: i< j\}\subseteq  [2, 2^{k+1}]$ is distinct from $d_1x_1+\dots+d_kx_k\in  [-\Delta2^{k+1}, -2]$ because  $ [2, 2^{k+1}] \cap [-\Delta2^{k+1}, -2]=\emptyset$  due to $\Delta^{-1}, k^{-1}\gg n^{-1}$. 
\end{proof}
The following lemma finds vectors $x$ of length $\geq 3$ with $x\ast d$ simple.
\begin{lemma}\label{Lemma_xastd_simple_k_at_least_3} 
Let $k\geq 3$, $\Delta^{-1}, k^{-1}\gg n^{-1}$,  and $G$ an order $n$ abelian group. Let $(d_1,\dots, d_k)\in [-\Delta, -1]^k$ with each $d_i\not\equiv 0 \pmod G$. There exists some simple $x\in G^k$ with $x\ast d$  simple.
\end{lemma}
\begin{proof}
Pick $\Delta^{-1}, k^{-1}\gg m^{-1}\gg n^{-1}$.
If $G$ has a $\mathbb{Z}_s$-factor for some $s\geq m$, then use Lemma~\ref{Lemma_xastd_simple_cyclic} to get a simple $\hat x=(\hat x_1, \dots, \hat x_k)\in (\mathbb{Z}_s)^k$ with $\hat x\ast d$ simple. Construct $x=(x_1, \dots, x_k)\in G^{k}$  by letting each $x_i$ agree with $\hat x_i$ on the $\mathbb{Z}_s$-factor of $G$ and be zero on all other factors. Now $x$ is simple with $x\ast d$ simple (with required things distinct on the $s$th coordinate), and hence satisfies the lemma. So we can assume that $G=\mathbb{Z}_{m_1}\times \dots \times \mathbb{Z}_{m_t}$ with each $m_i\leq m$. In particular, we have that $t\geq m$ since $m^{-1}\gg n^{-1}$.

For each coordinate $i\in [k]$, define $f_i:x\to d_ix$. Say that $f_i$ is trivial on an abelian group $H$ if $f_i(x)=0$ for all $x\in H$.
Since $d_i\not\equiv 0\pmod G$, we have that for each $i$ there exists some $j(i)$ with $f_i$ non-trivial on $\mathbb{Z}_{m_{j(i)}}$. Let $S(i)$ be the set of such $j(i)$.

Suppose there are distinct $a,b,c\in [k]$ with distinct $m(a)\in S(a), m(b)\in S(b), m(c) \in S(c)$. 
Pick $x_a\in \mathbb{Z}_{m(a)}$ with $f_a(x_a)\neq 0$, and similarly for $b,c$.
For $i\in[k]\setminus\{a,b,c\}$ pick $m(i)\in [t]$ distinctly from each other and from $m(a), m(b), m(c)$ (there's space to do this since there are $t\gg k$ choices for each $m(i)$), and pick  $x_i$ to be anything in  $\mathbb{Z}_{m(i)}$. The resulting $x$ is simple and has $x\ast d$ simple since all the elements of $\{x_1, \dots, x_k\}\cup\{x_i+x_j: i< j\}\cup \{d_1x_1+\dots+d_kx_k\}$  have different supports ($x_i$ is supported on $\{m(i)\}$, $x_i+x_j$ is supported on $\{m(i), m(j)\}$,  and $d_1x_2+\dots+d_kx_k$ is supported on some set containing $\{m(a), m(b), m(c)\}$. These are all distinct sets due to $m(1), \dots, m(k)$ being distinct).
Thus we can suppose that 
\begin{itemize}
\item[$(\ast)$] there do not exist  distinct $a,b,c\in [k]$ with distinct $m(a)\in S(a), m(b)\in S(b), m(c) \in S(c)$. 
\end{itemize}
\par We claim that $(\ast)$ implies that there is at most one $a\in [k]$ with $|S(a)|\geq 3$. Suppose the contrary. Then we have $a,b$ with $S(a), S(b)\geq 3$. Since $k\geq 3$ there's some $c\in [k]\setminus\{a,b\}$, and since $S(c)\geq 1$ always, we can pick $m(c)$ to be anything in $S(c)$. Next, since $S(a), S(b)\geq 3$, we can pick $m(a)\in S(a), m(b)\in S(b)$ distinctly from each other and from $m(c)$, giving a contradiction to $(\ast)$. 
\par Next, we claim $|\bigcup_{i\neq a} S(i)|\leq 6$. Suppose otherwise for sake of contradiction. Let $\{b_1, \dots, b_t\}\subseteq [k]\setminus \{a\}$ be a minimal set with $\bigcup_{i\neq a} S(i)=\bigcup_{i=1}^tS(b_i)$. Since each $|S(b_i)|\leq 2$, we have $t\geq 3$. But by minimality, for each $i\ne a$, there's some $m(b_i)\subseteq S(b_i)\setminus \bigcup_{j\neq i,a}S(b_i)$, contradicting $(\ast)$. 
\par Now, pick  $x_a\in G$ arbitrary (e.g. $x_a=0$). Let $F=\{d_1x_1+\dots +d_kx_k: x_1, \dots,x_{a-1}, x_{a+1}, \dots,  x_k\in G\}$, noting that $|F|\leq m^{2k}$ (we can assume that each $x_i$ is supported on $S(i)$, since $f_i$ is trivial on other coordinates. The number of choices of $x_i$ supported on $S(i)$ is $\leq m^{|S(i)|}\leq m^2$ since all cyclic factors in $G$ have size $\leq m$). Pick distinct $x_1, \dots, x_{a-1}, x_{a+1}, \dots, x_k\neq x_a$ one by one so that all sums $x_i+x_j$ are outside $F$ and distinct (when picking $x_i$   we need to ensure that $x_{i}\not\in  \bigcup_{j<i} (F-x_j)\cup\{x_r+x_s-x_{t}: r,s,t\in  [1, i-1]\cup \{a\}\}$ which has size $\leq k(m^{2k}+k^3)\ll n$). Now $x$ satisfies the lemma.
\end{proof}

The above lemma isn't true for $k=2$ since when $(d_1, d_2)=(1,1)$, there is no $x$ with $x\ast d$ simple. The following shows that this is the only exception.
\begin{lemma}\label{Lemma_xastd_simple_k_2}
Let $1\gg n^{-1}$,  and let $G$ be an order $n$ abelian group. Let $(d_1,d_2)\in G\times G$ with either $d_{1}\not\equiv 1\pmod G$ or $d_{2}\not\equiv 1\pmod G$. There exists some simple $x\in G^2$ with  $x\ast d$ simple.
\end{lemma}
\begin{proof}
Without loss of generality $d_{1}\not\equiv 1\pmod G$.
Let $f:G\times G\to G$ with $f(x_1,x_2)=(d_1-1)x_1+(d_2-1)x_2$. Since $d_{1}\not\equiv 1\pmod G$, this is not identically zero. It is also a homomorphism, and so, using Lagrange's Theorem, $|f^{-1}(0)|\leq |G\times G|/2=n^2/2$.
Thus $|f^{-1}(G\setminus 0)|\geq n^2/2$. 
Since there are exactly $n$ pairs $(x_1, x_2)$ with $x_1=x_2$, there exists some pair $(x_1, x_2)$ with $f(x_1, x_2)\neq 0$ and $x_1, x_2$ distinct. These satisfy the lemma. Indeed $(x_1, x_2)$ is simple since $x_1\neq x_2$, while $(x_1, x_2)\ast (d_1, d_2)$ is simple because $d_1x_1+d_2x_2\neq x_1+x_2$ (which is equivalent to $f(x_1, x_2)\neq 0$.
\end{proof}
The following lemma combines the previous two and characterizes when one can find some $x$ with $x\ast_T v$ simple.
\begin{lemma}\label{Lemma_kastv_simple_tree}
Let $\Delta^{-1}, k^{-1}\gg n^{-1}$, $G$ an order $n$ abelian group and $T$ a graph with $\Delta(T)\leq \Delta$. Let $v=(v_1, \dots, v_k)$ be a sequence of distinct vertices in $T$ having $d_T(v_i)\not\equiv 1 \pmod G$. Then there exists some simple $x\in G^{k}$ with $x \ast_T v$ simple unless:
\begin{itemize}
\item[$(\ast)$] $k=2$, $d_T(v_1), d_T(v_2)\equiv 0 \pmod G$, and $v_1v_2\in E(T)$.
\end{itemize}
\end{lemma}
\begin{proof}
Set $d=(-d_T(v_1)+1, \dots, -d_T(v_k)+1)$, noting that for all $i$, we have $d_i\not \equiv 0 \pmod G$ and $d_i\in [-\Delta+1, -1]$ (since $d_T(v_i)\not\equiv 1 \pmod G$ and $\Delta(T)\leq \Delta$).

If $k=1$, pick $x_1$ to be anything in $G$, noting that then the multiset $x \ast_T v$ contains only one element in total (namely $(-d_T(v_1)+1)x_1$), and hence is simple.

If $k\geq 3$, then the result follows by a direct application of Lemma~\ref{Lemma_xastd_simple_k_at_least_3}. 

If $k=2$, and $v_1v_2\not\in E(T)$, pick $x_1, x_2$ arbitrary distinct elements of $G$. Then the multiset $x \ast_T v$ contains only one element in total (namely $(-d_T(v_1)+1)x_1+(-d_T(v_2)+1)x_2$), and hence is simple.

If $k=2$, and $v_1v_2\in E(T)$, then, since $(\ast)$ doesn't hold we have that  $d_T(v_1)\not\equiv 0 \pmod G$ or $d_T(v_2)\not\equiv 0 \pmod G$. Now the result follows by a direct application of Lemma~\ref{Lemma_xastd_simple_k_2}
\end{proof}

It's well known that in every group, other than $\mathbb{Z}_2^k$, there are two distinct elements summing to $0$. The following lemma shows that we can get such elements with a prescribed sum as well.
\begin{lemma}\label{Lemma_y1y2_separate}
Let $G\neq \mathbb{Z}_2^k$ be an order $n$ abelian group. Then, for any $g\in G$, there are $> n/2$ solutions to $y_1+y_2=g$ with $y_1\neq y_2$. In particular, for any $F_1\subseteq G$ with $|F_1|< n/4$, there is such a solution with $y_1, y_2\not\in F_1$.
\end{lemma}
\begin{proof}
Since $G\neq \mathbb{Z}_2^k$, the number of solutions to $2y=0$ is $< n$. Since the set of such solutions forms a subgroup, the number of these solutions is actually $\leq n/2$. The number of solutions to $2y=g$ is either zero or equals the number of solutions to $2y=0$ (given one element $y'$ with $2y'=g$, for any other $y$ with $2y=g$ we have $2(y-y')=0$). Thus, in either case, the number of solutions to  $2y=g$ is $\leq n/2$. 

Let $Y=\{y\in G: 2y\neq g\}$, noting that we have established $|Y|> n/2$. Note that for all $y\in Y$, we have that $(y_1=y, y_2=g-y)$ is a solution to $y_1+y_2=g$ with $y_1\neq y_2$ --- thus we have established that there are $>n/2$ such solutions.

For the ``in particular part'', note that at most $n/4$ of the identified solutions can have $y_1\in S$, and at most $n/4$ of them can have $y_2\in S$, leaving at least one with $y_1, y_2\not\in F_1$.
\end{proof}

The following lemma is similar to the above, but deals with sums of more than two elements.
\begin{lemma}\label{Lemma_y1y2...ys_separate}
Let $C^{-1}, s^{-1}\gg n^{-1}$ with $s\geq 3$.
Let $G$ be an abelian group. Then, for any $g\in G$ and $F_1, F_2\subseteq G$ with $|F_1|, |F_2|\leq C$, there is a solution to $y_1+y_2+\dots+y_s=g$ with $y_i\not\in F_1, y_{i}-y_j\not\in F_2$ for all $i\neq j$.
\end{lemma}
\begin{proof}
There are $n^{s-1}$ solutions to $y_1+y_2+\dots+y_s=g$. For any fixed $i\in[s], f\in F_1$, there are $n^{s-2}$ solutions to $y_1+y_2+\dots+y_s=g$ with $y_i=f$ (these are exactly the solutions to $y_1+\dots+y_{i-1}+ y_{i+1}+\dots+y_s=g-f$). For any distinct $i,j\in[s], f\in F_2$, we claim that there are $n^{s-2}$ solutions to $y_1+y_2+\dots+y_s=g$ with $y_i-y_j=f$. To see this, note that without loss of generality, we may assume $i=1,j=2$. Now, first pick $y_1$, for which there are $n$ choices. Afterwards, we are looking for $y_3, \dots, y_s$, which satisfy $y_3+\dots+y_s=g+f-2y_1$, for which there are exactly $n^{s-3}$ solutions --- and this is where we are using that $s\geq 3$. Thus, in total, we have $n\times n^{s-3}=n^{s-2}$ solutions, as claimed. 
\par We may thus conclude that there are at least $n^{s-1}-(s|F_1|+\binom s2|F_2|)n^{s-2}>1$ solutions satisfying the lemma.
\end{proof}

We'll need the following lemma for the case when our group is $\mathbb{Z}_2^k$.
\begin{lemma}\label{Lemma_zero_sum_sets_Z2k}
Let $T$ be a graph whose vertex set is partitioned $V(T)=A\cup B$. Suppose that $k\gg |A|+|B|\ge 10$, $|A|, |B|\ne 2$ and if $|A|=4$ or $|B|=4$ then $T[A]$ or $T[B]$ has no perfect matching (respectively). Then, there exists a rainbow embedding $\phi: T \to \mathbb{Z}_2^k$ also satisfying that $\sum \phi(A)=\sum \phi(B)=0$.
\end{lemma}
\begin{proof}
Set $a:=|A|$, $b:=|B|$. 
Let $e_i\in \mathbb{Z}_2^k$ denote the vector with $1$ in the $i$th $\mathbb{Z}_2$-factor and zeros everywhere else.
Without loss of generality, we have that $b\geq a$ which implies $b\geq 5$. If $a=1$, pick $\phi(A)=\{(0, \dots, 0)\}$, otherwise pick $\phi(A)=\{e_1, \dots, e_{a-1}, e_1+\dots + e_{a-1}\}$ (noting that this is a set of order $a$ using that $a\ne 2$).
Pick $\phi(B)=\{e_{a+1}, \dots, e_{a+b}, e_{a+1}+ \dots+ e_{a+b}\}$. There's space to pick $A,B$ like this since $k\gg a,b$.
We have that $A,B$ are disjoint since all listed elements of $A,B$ have distinct supports (using that $a,b\neq 2$). We have $\sum \phi(A)=2e_1+\dots + 2e_{a-1}=0$ and $\sum \phi(B)= 2e_{a+1}+ \dots+ 2e_{a+b}=0$.
To see that the embedding is rainbow: note that for distinct $\{x,y\}, \{z,w\}\subseteq A\cup B$,  we have that $\phi(x)+\phi(y)$ and $\phi(z)+\phi(w)$ have distinct supports unless $\{x,y,z,w\}= A$ or $B$. This could only stop $\phi(T)$ from being rainbow if $T[A]$ or $T[B]$ had order $4$ and had a perfect matching --- which is excluded in the lemma's assumption.
\end{proof}

We now prove the main result of this section, i.e. Theorem~\ref{Theorem_main_intro}, phrased in the following equivalent formulation.
\begin{theorem}\label{Theorem_harmonious_tree_characterization}
Let $T$ be a tree with $\Delta(T)\leq \Delta$  and $G$ an abelian group. There is a rainbow copy of $T$ in $K_G$ if, and only if, we have none of the following:
\begin{enumerate}[(1)]
\item $G=\mathbb{Z}_2^m$, $m\geq 2$  and  $T$ is a path (or equivalently $T$ has precisely two vertices of odd degree).
\item $G=\mathbb{Z}_2^m$, $m\geq 2$  and  $T$ has precisely two vertices of even degree.
\item $G=\mathbb{Z}_{m_1}\times\dots \times \mathbb{Z}_{m_k}$ and $V(T)=\{v_1, \dots, v_n\}$, with $v_1v_2\in E(T)$ and  $d_T(v_1), d_T(v_2)\equiv 0\pmod {m_i}$, $d_T(v_3), \dots, d_T(v_n)\equiv 1\pmod {m_i}$ for all $i$.
\item $G=\mathbb{Z}_2^m$, $m\geq 2$ , $T$ contains precisely $4$ vertices of  even degree and has a perfect matching when restricted to these $4$ vertices.
\end{enumerate}
\end{theorem}
\begin{proof}

``Only if'' direction:\\
Suppose that $G=\mathbb{Z}_2^m$. Let $V_{odd}, V_{even}$ be the sets of odd/even degree vertices of $T$.
Suppose that there is a rainbow embedding $\phi$ of $T$ into $V(G)$. Then we must have $C(\phi(T))=G\setminus \{0\}$ since the colour $0$ doesn't appear on any edges of $K_{\mathbb{Z}_2^m}$.
Then $\sum_{v\in V(T)}\phi(v)=\sum G=0$ (using that $m\geq 2$) and $\sum_{v\in V(T)}d_T(v)\phi(v)=\sum C(\phi(T))=\sum G-0=0$. Since $\sum_{v\in V_{odd}}\phi(v)=\sum_{v\in V_{odd}}d_T(v)\phi(v)=\sum_{v\in V(T)}d_T(v)\phi(v)$ and $\sum_{v\in V_{even}}\phi(v)= \sum_{v\in V(T)}\phi(v)-\sum_{v\in V_{odd}}\phi(v)$, we get that $\sum_{v\in V_{odd}}\phi(v), \sum_{v\in V_{even}}\phi(v)=0$. Since all the vertices of $\phi(V_{even}), \phi(V_{odd})$ must be distinct, this means that $|V_{odd}|, |V_{even}|\neq 2$ (because in $\mathbb{Z}_2^m$ we cannot have two distinct elements adding to $0$). If $|V_{even}|=4$, then we get that $T[V_{even}] $ doesn't have a perfect matching, since otherwise the two edges of this matching must have the same colour in the embedding.

Suppose that ``$G=\mathbb{Z}_{m_1}\times\dots \times \mathbb{Z}_{m_k}$ and $V(T)=\{v_1, \dots, v_n\}$, where $d_T(v_3), \dots, d_T(v_n)\equiv 1\pmod {m_i}$ for all $i$, and $v_1v_2\in E(T)$ and $d_T(v_1), d_T(v_2)\equiv 0\pmod {m_i}$ for all $i$''.
Suppose for contradiction that there is a rainbow embedding $\phi$ of $T$ into $V(G)$. Let $c$ be the unused colour. We have $\sum G=\sum_{v\in V(T)} \phi(v)$ and $\sum G-c=\sum_{v\in V(T)} d_T(v)\phi(v)$. Subtracting gives $c=\sum_{v\in V(T)} (1-d_T(v))\phi(v)$. Since $d_T(v_3), \dots, d_T(v_n)\equiv 1\pmod {m_i}$ and $d_T(v_1), d_T(v_2)\equiv 0\pmod {m_i}$ for all $m_i$ we have that $(1-d_T(v_i))\phi(v_i)=0$ in $G$ for $i=3, \dots, m$ and $(1-d_T(v_1))\phi(v_1)=\phi(v_1)$, $(1-d_T(v_2))\phi(v_2)=\phi(v_2)$.
 This gives  $c=\sum_{v\in V(T)} (1-d_T(v))\phi(v)=(1-d_T(v_1))\phi(v_1)+(1-d_T(v_2))\phi(v_2)=\phi(v_1)+\phi(v_2)$. But $\phi(v_1)+\phi(v_2)$ is also the colour of the edge $v_1v_2$ contradicting that $c$ is not used on $\phi(T)$.

``If'' direction:\\
Suppose $G\neq \mathbb{Z}_2^m$. If $T$ is a path, let $T_{core}$  be an independent set consisting  of both leaves, and $6$ vertices of degree $2$ (noting that this is a core of $T$). Otherwise, let  $T_{core}$ be a core of $T$ of size $\leq 12\Delta$ given by Observation~\ref{Observation_small_core}, noting that $T_{core}$ will then contain $\geq 3$ leaves of $T$ (since $T$ is not a path it has $\geq 3$ leaves. Now depending on whether I or II occurs, $T_{core}$ contains either all the leaves or at least $6$  leaves). Label $V(T_{core})=\{v_1, \dots, v_k, w_1, \dots, w_t, u_1, u_2, \dots, u_s\}$ where
for all $i$, $d(v_i)\not\equiv 1 \pmod G$,  $d(w_i)\equiv 1 \pmod G$ with $d(w_i)\neq 1$, and $u_1, \dots, u_s$ are leaves.  
Let $T_{core}^{v}=T[\{v_{1}, \dots, v_k\}]$ and
$T_{core}^{w}=T[\{v_{1}, \dots, v_k, w_1, \dots, w_t\}]$

\begin{claim}\label{Claim_harmonious_tree_characterization1}
There is a  $\psi:\{v_1, \dots, v_k\}\to V(K_G)$ which is a rainbow embedding of $T_{core}^v$ not using the colour $c_{special}:=(-d(v_1)+1)\psi(v_1)+\dots+(-d(v_k)+1)\psi(v_k)$.
\end{claim}
\begin{proof}
Note that if  $k=2$, then $v_1v_2\not\in E(T)$ --- otherwise the degrees of all vertices in $T_{core}$ other than $v_1,v_2$ must be $\equiv 1 \pmod G$. This would imply that the degrees of all vertices in $T$ other than $v_1,v_2$ are $\equiv 1 \pmod G$ (since $T_{core}$ has a representative vertex of every degree occurring in $T$ by the definition of ``core''), and hence (3) would hold.

Thus we can apply Lemma~\ref{Lemma_kastv_simple_tree} we get some simple $x\in \mathbb{G}^k$ with  $x\ast_T d$ simple. Define $\psi$ to embed $v_i$ to $x_i$ for all $i$. Since $x$ is simple, this is an injection. 
The multiset of colours it uses is $\{x_i+x_j: v_iv_j\in E(T)\}=x\ast_T d\setminus\{c_{special}\}$. Since $x\ast_T d$ is simple, we get that these colours are all distinct from each other and from $c_{special}$.
\end{proof}

\begin{claim}\label{Claim_harmonious_tree_characterization2}
We can extend $\phi$ to  $\theta:\{v_1, \dots, v_k, w_1, \dots, w_t\}\to V(K_G)$ which is a rainbow embedding of $T_{core}^w$ not using the colour $c_{special}$.
\end{claim}
\begin{proof}
For $i=0, \dots, t$ set $V_i=\{v_1, \dots, v_k, w_1, \dots, w_i\}$. 
We build rainbow embeddings $$\theta_i:\{v_1, \dots, v_k, w_1, \dots, w_i\}\to V(K_G)$$ one by one. 
Start with $\theta_0=\psi$. To build $\theta_{i+1}$ from $\theta_{i}$:
Pick $\theta_{i+1}(w_{i+1})$ to be anything outside $\theta_{i}(V_{i})\cup (\theta_{i}(V_{i})+\theta_{i}(V_{i})-\theta_{i}(V_{i}))\cup (c_{special}-\theta_{i}(V_{i})$ (there's space to do this since $|\theta_{i}(V_{i})\cup (\theta_{i}(V_{i})+\theta_{i}(V_{i})-\theta_{i}(V_{i}))\cup (c_{special}-\theta_{i}(V_{i}))|\leq (k+t)+(k+t)^3+(k+t)\leq 3|V(T_{core})|^3\ll n$). 
Note that $\theta_{i+1}$ is an injection since $\theta_{i}$ was one, and $\theta_{i+1}(w_{i+1})\not\in \theta_{i}(V_{i})$. Also $\theta_{i+1}$ is a rainbow embedding since $\theta_{i}$ was one, and the new colours used by $\theta_{i+1}$ are contained in $\theta_{i+1}(w_{i+1})+\theta_{i}(V_i)$ which is disjoint from $C(\theta_i(T[V_i]))\subseteq \theta_{i}(V_i)+\theta_{i}(V_i)$ (due to $\theta_{i+1}(w_{i+1})\not\in (\theta_{i}(V_i)+\theta_{i}(V_i)-\theta_{i}(V_i))$). 
Finally, $\theta_{i+1}$ doesn't use the colour $c_{special}$ since $\theta_{i+1}(w_{i+1})+\theta_{i}(V_i)$ is disjoint from $\{c_{special}\}$ (this is equivalent to $\theta_{i+1}(w_{i+1})\not\in (c_{special}-\theta_{i}(V_i))$).
\end{proof}

\begin{claim}\label{Claim_harmonious_tree_characterization3}
We can extend $\theta$ to  $\phi:\{v_1, \dots, v_k, w_1, \dots, w_t, u_1, \dots, u_s\}\to V(K_G)$ which is a rainbow embedding of $T_{core}$ not using the colour $c_{special}$ and satisfying $\sum V(\phi(T_{core}))=\sum G$.
\end{claim}
\begin{proof}
Recall $u_1, \dots, u_s$ are leaves with $s\geq 2$. Let $N:=\bigcup_{i=1}^sN_{T}(u_i)\cap V(T_{core})$, noting that when $s=2$, we have ensured $N=\emptyset$. Let $g:=\sum G- \sum im_{\theta}$, $F_1= im_{\theta}\cup (im_{\theta}+im_{\theta}-im_{\theta})\cup (c_{special} -im_{\theta})$ and $F_2=N-N$, noting that $|F_1|, |F_2|\leq 3|T_{core}|^3\leq 3(12\Delta)^3\ll n$ and that when $s=2$ we have $F_{2}=\{0\}$. 
Depending on whether $s=2$ or not, use Lemma~\ref{Lemma_y1y2_separate} or~\ref{Lemma_y1y2...ys_separate} to pick  $y_1, \dots, y_s\not\in F_1$ with $y_1+\dots+y_s=g$ and $y_i-y_j\not\in F_2$ for $i\neq j$. 
Define $\phi$ to agree with $\theta$ on $\{v_1, \dots, v_k, w_1, \dots, w_t\}$ and to have $\phi(u_i) = y_i$. 
This is an injection because $y_1, \dots, y_s$ are distinct and outside $im_{\theta}$. When $s=2$, there are no edges in $T_{core}$ touching $u_1, \dots, u_s$, so we have a rainbow embedding in that case.  When $s\geq 3$, the  colours of new edges used by $\phi$ (i.e. the colours edges $\phi(xy)$ with $xy\not\in T_{core}^w$) are contained in $\{im_{\theta}+y_i: i=1, \dots, s\}$ (here, we're using that $u_1, \dots, u_s$ is an independent set due to it being a set of leaves of a tree $T$). These colours are all distinct from each other (since $y_i-y_j\not\in im_{\theta}-im_{\theta}$), from the colours of $\theta(T_{core}^w)$ (since $y_i\not \in im_{\theta}+im_{\theta}-im_{\theta}$), and from $c_{special}$ (since  $y_i\not \in c_{special}-im_{\theta}$).
Finally we have $\sum_{v\in V(T_{core})}\phi(v)
=\sum im_{\theta} +\sum_{i=1}^s y_i
=\sum G$
\end{proof}
Set $C_{target}=C(K_G)\setminus\{c_{special}\}$ and $V_{target}=V(K_G)$, noting that $\sum C_{target}=\sum G-c_{special}$ and $\sum V_{target}=\sum G=\sum V(\phi(T_{core}))$. Using that vertices  $v\in V(T_{core})\setminus\{v_1, \dots, v_k\}$ have $d_T(v)\equiv 1 \pmod G$, we get
\begin{align*}
\sum_{v\in V(T_{core})}d_T(v)\phi(v)&= \sum_{v\in V(T_{core})}(d_T(v)-1)\phi(v)+\sum_{v\in V(T_{core})}\phi(v)= \sum_{v\in V(T_{core})}(d_T(v)-1)\phi(v)+\sum G\\
&=\sum_{i=1}^{k}(d_T(v_k)-1)\phi(v_k)+\sum G
=\sum G-c_{special}=\sum C_{target}
\end{align*}
 Thus the embedding $\phi$ satisfies Theorem~\ref{Theorem_embed_tree_prescribed_sets} (ii), and hence we get a rainbow embedding of $T$ into $(V_{target}, C_{target})$ (and hence into $K_G$).

Suppose $G= \mathbb{Z}_2^m$. Use Observation~\ref{Observation_small_core} to get a core $T_{core}$ of $T$ of order $12\Delta$. Set $C_{target}=C(K_G)\setminus\{0\}, V_{target}=V(K_G)$, noting that these both have zero sum (since $\sum \mathbb{Z}_2^m=0$ for $m\geq 2$).
Let $A=\{v_1, \dots, v_a\}$ be the odd degree vertices in $T_{core}$ and $B=\{u_1, \dots, u_b\}$ the even degree vertices. Note that $|A|\neq 2$, as otherwise the degrees $d(v_1)$ and $d(v_2)$ must be exhausted in $T_{core}$ --- since non-exhausted degrees $d$ have $\geq 3$ degree $d$ vertices in every core and thus $v_1, v_2$ are the only odd degree vertices in $T$, and hence $T$ is a path, contradicting (1) not holding. Similarly $|B|\neq 2$ --- since otherwise $u_1$ and $u_2$ would be the only even degree vertices in $T$, contradicting (2) not holding.
 Also note that $|A|+|B|=|V(T_{core})|\geq 10$. If $|A|=4$, note that $T[A]$ can't have  a perfect matching (since leaves can't be connected in a $\ge 3$-vertex tree, for $T[A]$ to have a perfect matching, $A$ must have $\le 2$ leaves. But the only tree with $\le 2$ leaves is a path which doesn't have $4$ odd degree vertices). If $|B|=4$, note that $T[B]$ can't have  a perfect matching, as otherwise we'd have (4).
 
 Use Lemma~\ref{Lemma_zero_sum_sets_Z2k} to get a rainbow embedding $\phi$ of $T[A\cup B]$ with $\sum \phi(A)=\sum \phi(B)=0$. This ensures that $\sum_{v\in V(T_{core})}d_T(v)\phi(v), \sum \phi(V(T_{core}))=0=\sum V_{target}=\sum_{C_{target}}$ and hence by Theorem~\ref{Theorem_embed_tree_prescribed_sets},  we get a rainbow embedding of $T$ in $K_G$.
\end{proof}

\section{Concluding remarks}\label{sec:conc}
\textit{Hovey's cordial labelling conjecture.} Hovey \cite{hovey1991cordial} conjectured that the vertices of all trees can be labelled by $\mathbb{Z}_k$ (for any $k$) so that each label occurs either $s$ or $s+1$ times for some $s$, and furthermore, labelling the edges by the sum of the labels of their endpoints, each label occurs either $t$ or $t+1$ times for some $t$. Taking $k$ to be the number of the edges of the tree, we can see that Hovey's conjecture generalises the Graham--Sloane conjecture. The methods of the present paper can confirm Hovey's conjecture for  $k\gg \Delta$ and all trees with $\Delta(T)\le \Delta$, although a formal proof would require a slight strengthening of Theorem~\ref{Theorem_embed_tree_prescribed_sets} to allow for embedding a few vertices of unbounded (but at most logarithmic) degree, which leads to some undesirable technicalities, hence we do not provide details here. 

\par \textit{The Graham--H\"aggkvist conjecture}. A well-known conjecture of Graham and H\"aggkvist~\cite{Häggkvist_1989}, which can be interpreted as a natural bipartite analogue of Ringel's conjecture, is the following.
\begin{conjecture}[The Graham--H\"aggkvist conjecture]
    Any $n$-edge tree decomposes the edge set of the balanced complete bipartite graph $K_{n,n}$.
\end{conjecture}
Although Ringel's conjecture has been resolved for large $n$ \cite{ringel, keevash2025ringel}, the Graham--H\"aggkvist conjecture is still open. As an approach to the Graham--H\"aggkvist conjecture, Ringel and Llad\'o (see \cite{camara2009conjecture} and the references therein) made the following conjecture that can be considered a bipartite version of the graceful tree conjecture, i.e. the Ringel--Kotzig conjecture. A bigraceful labeling of a tree $T$ with $n$ edges and bipartition $(A,B)$
is a map $\phi$ of $V(T)$ on the integers $[m-1]$ such that the restriction of $\phi$ to each of $A$ and $B$ is injective and the values $\phi(u)-\phi(v)$ for each edge $u,v$ is pairwise distinct and must be contained in $[m-1]$. The Ringel--Llad\'o conjecture would imply the Graham--H\"aggkvist conjecture by way of cyclic translations, see \cite{camara2009conjecture}. Here we propose a different conjecture that would also imply the Graham--H\"aggkvist conjecture which might be more approachable due to more slack in the choice of the labels.
\begin{conjecture}\label{anotherrainbowconjecture} Let $T$ be a $n$-edge tree. Consider an edge-coloured bipartite graph between two copies of $\mathbb{Z}_n$, say $(A,B)$, where the colour of an edge $(a,b)\in A\times B$ is $b-a\in \mathbb{Z}_n$. Then, there exists a rainbow embedding of $T$.
\end{conjecture}
To see how the above conjecture would imply the Graham--H\"aggkvist conjecture, we simply consider cyclic translations of a rainbow tree as in Observation~\ref{obs:cyclicshift}. Given a rainbow $n$-edge tree $T$, $x+T$ denotes the translated isomorphic rainbow tree obtained by replacing each vertex $v_A$ of $T$ of part $A$ with the vertex $x+v_A\in A$, and each vertex $v_B$ of part $B$ with the vertex $x+v_B\in B$. As the colour of each edge is preserved in the translation, the $n$ possible translations by elements of $\mathbb{Z}_n$ gives the decomposition required by the Graham--H\"aggkvist conjecture. 
\par We believe that the methods in the current paper with little modifications would confirm Conjecture~\ref{anotherrainbowconjecture} for bounded degree trees, and therefore the Graham--H\"aggkvist conjecture for bounded degree trees as well. However, we do not include further details in the present paper, as handling the sum-based colouring rule (as required by the Graham--Sloane conjecture) and the difference based colouring rule (as in Conjecture~\ref{anotherrainbowconjecture}) with a unified proof would lead to some undesirable technicalities, see for example \cite[Section 4]{muyesser2022random}. 
\par The high degree case of Conjecture~\ref{anotherrainbowconjecture} may be more approachable than the high degree case of the Graham--Sloane conjecture, as the host graph has $2|T|$ vertices, reminiscent of the set-up in the proof of Ringel's conjecture from \cite{ringel}.
\par \textit{The oriented rainbow tree conjecture}. We propose the following conjecture to unify several rainbow-type problems in combinatorics. 
\begin{conjecture}[The oriented rainbow tree conjecture]\label{conj:directedrainbow} There exists an absolute constant $C$ such that the following holds.
    Let $D$ be any $d$-regular properly coloured digraph\footnote{Precisely, a digraph where each vertex has $d$ in and $d$ out edges, loops or parallel edges are not allowed, but $a\to b$ and $b\to a$ can both be edges. Properly coloured means that the colour of the in-edges of any vertex are all distinct, similarly for the out-edges. }. Let $T$ be any oriented tree on $d-C$ edges. Then, there is a rainbow copy of $T$ in $D$. 
\end{conjecture}

The above conjecture can be interpreted as a sparse version of the Montgomery--Pokrovskiy--Sudakov conjecture \cite[Conjecture 11.1]{approximateringel} generalised to digraphs. The conjecture is also related to a problem of Crawford, Sankar, Schildkraut, and Spiro~\cite{crawford2025rainbow} that asks if any $(d-C)$-edge tree has a rainbow embedding to any properly coloured graph with minimum degree $d$, where $C$ is an absolute constant. In a previous version of this article, we conjectured that $C$ can be taken to be $1$, but Example~\ref{ex:new} shows that $C$ has to be at least $2$ for a positive answer. On the other hand, if $T$ is a directed path, it is feasible that $C$ can be taken to be $1$. Whether this stronger statement is true is asked explicitly in \cite[Problem 1.6]{bucic2025towards}. Such a statement would imply that for any subset $S$ of any group $G$, $S$ can be permuted as $s_1,\ldots, s_k$ so that the partial products $s_1$, $s_1s_2$, $\ldots$, $s_1s_2\cdots s_k$ are all distinct. When $G=\mathbb{Z}_p$ for $p$ prime, this implies Graham's rearrangement conjecture \cite{graham1971sums} (reiterated by Erd\H{o}s and Graham in \cite{ErdosGraham}, see also \cite{bucic2025towards, BedertBucicKravitzMontgomeryMuyesser2025, BederdKravitz}), which is a long-standing open problem.
\par A positive solution to \cite[Problem 1.6]{bucic2025towards} (i.e. Conjecture~\ref{conj:directedrainbow} holds for directed paths with $C=1$) also would recover a recent breakthrough of Montgomery on finding large transversals in Latin squares.  The Ryser--Brualdi--Stein \cite{ryser1967neuere, brualdi1991combinatorial, stein1975transversals} conjecture asserts that any Latin square of order $n$ has a transversal of size $n-1$. This difficult conjecture was recently resolved for large $n$ by Montgomery \cite{Montgomery2024}. Latin squares are in one to one correspondence with $1$-factorisations of complete digraphs (with loops allowed) \cite{Pokrovskiy2022_RainbowSubgraphs}, and therefore upon the deletion of a colour class (corresponding to the self-loops), yield $n$-vertex, $(n-1)$-regular properly coloured digraphs. Then, a positive resolution to \cite[Problem 1.6]{bucic2025towards} would imply that any such digraph contains a $(n-2)$-edge directed path. This then implies that the Latin square contains a transversal of size $n-2$ (which is cycle-free in the sense of \cite{gyarfas2014rainbow}, therefore this already implies a conjecture of Gy{\'a}rf{\'a}s and S{\'a}rk{\"o}zy). For the (unique) vertex $v$ not included in the directed path, we may add back the edge corresponding to the self-loop on $v$ (whose colour was excluded on the path), we even get a transversal of size $n-1$, recovering Montgomery's theorem \cite{Montgomery2024}.
\par Conjecture~\ref{conj:directedrainbow} also has a strong connection with tree decompositions. Ringel's tree-decomposition conjecture reduces to embedding a rainbow copy of a $n$-edge tree on $\mathrm{ND}_{2n+1}$ (the near-distance colouring graph), where the vertices correspond to vertices of a regular $2n+1$ vertex polygon, and edge-colour corresponds to Euclidian distance (see \cite{ringel} for details). By orienting each edge clockwise, we obtain $\vec{\mathrm{ND}_{2n+1}}$, a $n$-regular digraph. Conjecture~\ref{conj:directedrainbow} then implies that $\mathrm{ND}_{2n+1}$ contains a rainbow copy of any $(n-C)$-edge tree, essentially recovering \cite{ringel}, which shows that $\mathrm{ND}_{2n+1}$ contains any rainbow $n$-edge tree.
\par We remark that all of the applications we mention here use directed rooted trees, i.e. orientations of trees obtained from picking a root, and directing all edges away from the root, so this special case of the conjecture is already fairly interesting. In fact, even for undirected graphs, giving an asymptotic solution to Conjecture~\ref{conj:directedrainbow} is an open problem. It would already be interesting to determine if all $d$-regular properly $d$-edge-coloured graphs contain any tree with $d-o(d)$ edges as a rainbow subgraph.

\bibliographystyle{abbrv}
\bibliography{harmonious}
\newpage

\section{Appendix: approximate tree embeddings}
Here we prove Lemma~\ref{Lemma_approximate_embedding}. The methods we use are all standard and taken from \cite{approximateringel, ringel}. 
We say a set of subtrees $T_1,\ldots, T_\ell\subset T$ divides a tree $T$ if $E(T_1)\cup\ldots \cup E(T_\ell)$ is a partition of $E(T)$. We use the following lemma.
\begin{lemma}[\cite{randomspanningtree}, Proposition 3.22]\label{Lemma_littletree} Let $n,m\in \mathbb N$ satisfy $1\leq m\leq n/3$. Given any tree~$T$ with~$n$ vertices and a vertex $t\in V(T)$, we can find two trees $T_1$ and $T_2$ which divide~$T$ so that $t\in V(T_1)$ and $m\leq |T_2|\leq 3m$.
\end{lemma}

Using this we can divide a forest into very small subtrees.
\begin{lemma}\label{Lemma_split_tree_small_components}
For any $m\in [1,n/10]$ and forest $T$, there is a set $I\subseteq V(T)$ of size $3n/m$  so that the connected components of $T\setminus I$ have size $\leq m$.
\end{lemma}
\begin{proof}
First, we prove the statement for trees $T$.
We do this by induction on $|T|$. In the initial case when $|T|\leq m$, take $I=\emptyset$ and there is nothing to prove. So suppose that $|T|>m$ and that the lemma holds for smaller $T$. Apply Lemma~\ref{Lemma_littletree} to divide $T$ into $T_1$ and $T_2$ so that $m/3\leq |T_2|\leq m$. Let $v$ be the (unique) common vertex of $T_1, T_2$. Apply induction to $T_1$ in order to find a set $I$ with $|I|\leq 3|T_1|/m\leq 3(n-m/3)/m=3n/m-1$ so that $T_1\setminus I$ has connected components smaller than $m$. Now $I\cup \{v\}$ satisfies the lemma.

When $T$ is a forest, let $T=T_1\cup \dots \cup T_k$, where $T_1, \dots, T_k$ are the connected components of $T$. Applying the connected statement to each $T_i$, we get subsets $I_i\subseteq T_i$ of size $\leq 3|V(T_i)|/m$, so that the components of $T_i\setminus I_i$ have size $\leq m$. Now $I:=I_1\cup \dots \cup I_k$ satisfies the lemma.
\end{proof}
The following analyses the structure of the forests $T\setminus I$ given by the above lemma.
\begin{lemma}\label{Lemma_decompose_forest_matchings}
Let $F$ be a forest with components of size $\leq m$. We can decompose $V(F)=V_0\cup V_1\cup \dots \cup V_m$ and $E(F)=M_1\cup \dots \cup M_m$ so that for each $i$, $M_i$ is  a matching from $V_i$ into $\bigcup_{j<i}V_j$.
\end{lemma}
\begin{proof}
Induction on $m$. In the initial case, $m=0$, we have that $F$ has no edges, so setting $V_0=V(F)$ works. Let $F$ be a forest with components of size $m$, and suppose the lemma holds for smaller $m$. Let $V_m$ be a set consisting of a degree $1$ vertex in each component of $F$ containing at least one edge, and let $M_m$ be the set of edges touching these vertices. Note that $M_m$ is a matching since all its edges are in different components.
We have that $F\setminus V_m$ has components of size $\leq m-1$, and hence by induction has a decomposition into $V_0\cup V_1\cup \dots \cup V_{m-1}$ and $M_1\cup \dots \cup M_{m-1}$. Adding $V_m, M_m$ to this decomposition gives one satisfying the lemma.
\end{proof}

The following is a small modification of the previous lemma.
\begin{lemma}\label{Lemma_decompose_forest_matchings_U1U2U3}
Let $F$ be a forest with components of size $\leq m$ and $V(F)=U_1\cup U_2\cup U_3$. We can decompose $V(F)=V_1\cup \dots \cup V_{3m+3}$ and $E(F)=M_{4}\cup \dots \cup M_{3m+3}$ so that for each $i\geq 4$, $M_i$ is  a matching from $V_i$ into $\bigcup_{j<i}V_j$ and also each $V_i\subseteq U_j$ for some $j$.
\end{lemma}
\begin{proof}
Let $V(F)=V_0\cup V_1\cup \dots \cup V_m$ and $E(F)=M_1\cup \dots \cup M_m$ be the decomposition from Lemma~\ref{Lemma_decompose_forest_matchings}. For $i=0, \dots, m$, $j=1,2,3$ set $V_i^j=V_i\cap U_j$ and let $M_i^j$ be the submatching of $M_i$ going from $V_i^j$ to $\bigcup_{t<i}V_t$. Now the sequences $V_0^1, V_0^2, V_0^3, V_1^1, V_1^2, V_1^3, \dots, V_m^1, V_m^2, V_m^3$ and $M_1^1, M_1^2, M_1^3, \dots, M_m^1, M_m^2, M_m^3$ satisfy the lemma (after suitably relabelling).
\end{proof}

The following allows embedding a single vertex in a rainbow manner. It is essentially the same as Lemma~\ref{Lemma_neighbourhood_prescribed_sum_random} (E2) --- and actually in the case when we are dealing with the graph $K_G$, we can simply replace all applications of the following lemma with Lemma~\ref{Lemma_neighbourhood_prescribed_sum_random} (E2).
\begin{lemma}\label{Lemma_embed_neighbourhood_random}
Let $\Delta^{-1}\gg \mu\gg \rho\gg n^{-1}$.
Let $K_n$ be properly edge-coloured, and $V,C$ independent $\geq n^{-\rho}$-random sets. With probability $\geq 1-o(n^{-1})$, for every $U\subseteq V(K_n)\cup C(K_n)$ with $|U|\leq n^{1-\mu}$ and every set $N\subseteq V$ of size $\leq \Delta$, there is a $C\setminus U$-common neighbour of $N$ in $V\setminus U$.
\end{lemma}
\begin{proof}
Without loss of generality we can assume that $V,C$ are $p$-random for $p=n^{-\rho}$ (if not, just pass to subsets of this probability). 
Fix a set $N$ of size $\leq \Delta$. For any vertex $v\in V(K_n)\setminus N$, we have $P(v \text{ is a $C$-common neighbour of $N$ in $V$})=p^{|N|+1}\geq p^{\Delta+1}$. This gives  that the expected number of $C$-common neighbours of $N$ in $V$ is $p^{|N|+1}(n-|N|)\geq p^{2\Delta}n/2$. This quantity is $\Delta$-Lipschitz, and so by Azuma's Inequality, with probability $1-o(n^{-2\Delta})$, it is $\geq p^{2\Delta}n/4$. Taking a union bound over all sets $N$, we have that with probability $\geq 1-o(n^{-1})$ all sets $N$ have $\geq  p^{2\Delta}n/4=n^{1-2\rho\Delta}/4$ $C$-common neighbours in $V$. Since $|U|\leq n^{1-\mu}<n^{1-2\rho\Delta}/4$, there is always one avoiding the colours/vertices of $U$.
\end{proof}


 For a $3$-uniform, $3$-partite hypergraph $H$, vertices $u,v$ and a subset $U\subseteq V(H)$, we define the \textbf{pair degree} of $(u,v)$ into $U$ as the number of vertices in $U$ which are in the neighbourhood of both $u$ and $v$, i.e. the number of vertices $z$ in $U$ such that there exists $v,w\in V(H)$ such that $\{u,z,v\}$ and $\{v,z,w\}$ are both edges of $H$. We say that $H$ is $(\gamma, p, n)$-\textbf{regular} if every part has $(1\pm \gamma)n$ vertices and every vertex has degree  $(1\pm \gamma)pn$. We say that $H$ is $(\gamma, p, n)$-\textbf{typical} if, additionally, every pair of vertices $x,y$ in the same part of $H$ have pair degree  $(1\pm \gamma)p^2n$ into every other part of $H$. We say that a hypergraph is \textbf{linear} if through every pair of vertices, there is at most one edge.

\begin{lemma}[\cite{muyesser2022random}, Lemma 3.8]\label{Lemma_2_random_1_deterministic}
Let $H=(A,B,C)$ be a tripartite linear hypergraph that is $(n^{-0.3},1,n)$-typical. Let $p\geq n^{-1/600}$ and let $A'\subseteq A$ be $p$-random, and let $B'$ a $p$-random subset of $B$, where $A'$ and $B'$ are not necessarily independent. Then, with probability at least $1-n^{-2}$, the following holds. 
\par For any $C'\subseteq C$ of size $(1\pm n^{-0.2})pn$, there is a matching covering all but $2n^{1-1/500}$ vertices in $A'\cup B'\cup C'$.
\end{lemma}
The following is a coloured-graph version of the above.
\begin{lemma}\label{Lemma_2_random_1_deterministic_graph}
Let $K_n$ be properly $n$-edge-coloured.
Let $p\geq n^{-1/600}$ and let $V\subseteq V(K_n)$, $C\subseteq C(K_n)$ be $p$-random, not necessarily independent. Then, with probability at least $1-n^{-2}$, the following holds. 
\par For any $U\subseteq V(K_n)\setminus V$ of size $\leq (1+ n^{-0.2})pn$, there is a $C$-rainbow matching into $V$ covering all but $2n^{1-1/500}$ vertices in $U$.
\end{lemma}
\begin{proof}
Let $C_{bad}$ be the set of colours appearing $<(n-n^{0.6})/2$ times. Note that then $|C_{bad}|(n-n^{0.6})/2+(n-|C_{bad}|)n/2\geq e(K_n)=\binom n2$, which is equivalent to $|C_{bad}|\leq n^{0.4}$. Let $K'$ be $K_n$ with edges of colours in $C_{bad}$ deleted, noting that all vertex degrees satisfy $n\geq d(v)\geq n-1-|C_{bad}|\geq n-2n^{0.4}$ and every colour appears $\geq (n-n^{0.6})/2$ times.
Define a tripartite hypergraph $H=(X,Y,Z)$ with $X=V(K_n), Y=C(K_n), Z=V(K_n)$, where $xyz$ is an edge whenever $xz$ is a colour $y$ edge of $K_n$. We have that $|X|,|Z|=n$ and $|Y|=n\pm n^{0.4}$,  $d_H(x)=d_{K'}(x)=n\pm 2n^{0.4}$ for $x\in X\cup Z$, and $d_H(y)=|\{v\in V(K'): \text{ there is a colour $y$ edge through $v$}\}|=n\pm 2n^{0.6}$ for $y\in Y$. Combining these, we obtain that $H$ is $(n^{-0.3},1,n)$-typical.

Letting $V'=V\cap X$ and $C'=C\cap Y$, we have that with probability  $\geq 1-n^{-2}$, $V', C'$ satisfy Lemma~\ref{Lemma_2_random_1_deterministic}. Consider now $U\subseteq V(K_n)\setminus V$ of size $\leq (1+ n^{-0.2})pn$. Add elements to $U$ to get a set $U'$ of size $(1\pm n^{-0.2})pn$. By Lemma~\ref{Lemma_2_random_1_deterministic}, there is a hypergraph matching $M$ covering all but $2n^{1-1/500}$ vertices in $V'\cup C'\cup U'$. Let $N$ be the set of edges in $K'$ corresponding to edges of $M$. Since $U\subseteq U'$, $N$ covers all but $2n^{1-1/500}$ vertices of $U$, and since $V'\subseteq V,C'\subseteq C$, these edges are $C$-coloured and go into $V$. They are rainbow because there's at most one edge of $M$ through each $y\in Y$, and they form a matching because $U, V$ are disjoint and there's at most one edge through each $x\in X, z\in Z$.
\end{proof}

The following is a version of the above which eliminates the need for having some vertices uncovered.
\begin{lemma}\label{Lemma_random_matching_into_independent_sets}
Let $1\gg \varepsilon \gg n^{-1}$ and $p\in [0,1-n^{-\varepsilon}]$.
Let $K_n$ be properly $n$-edge-coloured, and $V,C$ independent $(p+n^{-\varepsilon})$-random sets. With high probability, for every set $W\subseteq V(K_n)\setminus V$ with $|W|\leq p n$ there is a $C$-rainbow  matching from $W$ to $V$ which saturates $W$.
\end{lemma}
\begin{proof}
Pick $1\gg \mu\gg \varepsilon \gg n^{-1}$.
Partition $V= V_1\cup V_2,C=C_1\cup C_2$ where $V_1, C_1$ are $(p+n^{-\varepsilon}/2)$-random and $V_2, C_2$ are $n^{- \varepsilon}/2$-random.
With high probability $V_1, C_1$ satisfy Lemma~\ref{Lemma_2_random_1_deterministic_graph},  $V_2, C_2$ satisfy Lemma~\ref{Lemma_embed_neighbourhood_random} with $\Delta=1$, $\rho=\varepsilon$. 
Now consider some $W\subseteq V$ with $|W|\leq p n$.
 Apply Lemma~\ref{Lemma_2_random_1_deterministic_graph} to find a $C_1$-rainbow matching $M$ from $W$ to $V_1$ covering all but $k:=2n^{1-1/500}<n^{1-\mu}/2$ vertices in $W$. Let $w_1, \dots, w_k$ be the uncovered vertices in $W$. Repeatedly use the property of Lemma~\ref{Lemma_embed_neighbourhood_random} with $N=\{w_1\}, \{w_2\}, \dots, \{w_k\}$ to find a $C_2$-neighbour $v_i$ of each $w_i$. At the $i$th application, setting $U=\{v_1, \dots, v_{i-1}, c(w_1v_1), \dots, c(w_{i-1}v_{i-1})\}$ (which has size $\leq 2k\leq n^{1-\mu}$) ensures, the edges $w_1v_1, \dots, w_kv_k$ all have different vertices/colours. Now adding this matching to $M$ gives one satisfying the lemma.
\end{proof}

Now we prove the main result of the section.
\begin{lemma}
 Let $\Delta^{-1}\gg \varepsilon,\delta \gg n^{-1}$.
Let $K_{n}$ be properly $n$-edge-coloured and  $T$ a  forest with $\Delta(T)\leq \Delta$ and $|T|\leq (1-n^{-\delta})n$, and suppose we have a partition $V(T)=U_1\cup U_2\cup U_3$. 
Then there is a random  $f:V(T)\to K_n$ which is $n^{-\varepsilon}$-uniform on $\{V(T), U_1, U_2, E(T)\}$  and is a rainbow embedding of $T$ with high probability.   
\end{lemma}
\begin{proof}
Let $\Delta^{-1}\gg \gamma\gg \alpha\gg \beta\gg\varepsilon,\delta\gg n^{-1}$ (which implies $ \alpha\gg \varepsilon\Delta, \delta\Delta$ and   $n^{-\gamma}\ll n^{-\alpha}\ll n^{-\varepsilon\Delta}, n^{-\delta\Delta}$) and set $m:=n^{\alpha}$.
Let $I$ be the   set from Lemma~\ref{Lemma_split_tree_small_components} with $|I|\leq 3n^{1-\alpha}$, and set $I_i:=I\cap U_i$, $q_i:=|I_i|/n$ for $i=1,2,3$.
Let $F=T\setminus I$ to get a forest with components of size $\leq m$. Apply Lemma~\ref{Lemma_decompose_forest_matchings_U1U2U3} to get decompositions $V(F)=V_1\cup \dots \cup V_{3m+3}$ and $E(F)=M_4\cup \dots \cup M_{3m+3}$, where each $M_i$ is a matching which goes from $V_i$ to $\bigcup_{j<i}V_j$.
For each $i=1, \dots, 3m+3$, set $p_i:=|V_i|/n$ noting that $q_1+q_2+q_3+p_1+\dots+p_{3m+3}+3n^{-\beta} +(3m+3)n^{-\gamma}=|V(T)|/n+3n^{-\beta}+(3m+3)n^{-\gamma}\leq 1-n^{-\delta}+3n^{-\beta}+6n^{\alpha-\gamma}\leq 1$.

Pick disjoint random sets $R_V, R_{I_1}, R_{I_2}, R_{I_3}, R_1,\dots, R_{3m+3}, Q_1, \dots, Q_{3m+3}\subseteq V(K_n)$ and $$R_C, C_1, \dots, C_{3m+3}, D_1, \dots, D_{3m+3}\subseteq C(K_n)$$ where $R_i, C_i$ are $p_i$-random, $Q_i, D_i$ are $n^{-\gamma}$-random, $R_{I_i}$ are $q_i$-random, $R_V, R_C$ are $n^{-\beta}$-random, vertex sets are disjoint, colour sets are disjoint, and vertex/colour sets are independent.
With high probability Lemma~\ref{Lemma_embed_neighbourhood_random} applies to $R_V, R_C$ (with $\rho=\beta$ and $\mu=\alpha/2$), Lemma~\ref{Lemma_random_matching_into_independent_sets} applies to each pair $R_i\cup Q_i, C_i\cup D_i$ (with $p=p_i, \varepsilon=\gamma$), and the sizes of all sets are within $n^{1-\gamma}/2$ of their expectations.  Note that the last part gives $|R_i\cup Q_i|\geq (p_i+ n^{-\gamma})n-2n^{1-\gamma}/2=p_in=|V_i|$. 

For each $i=1, \dots, 3m+3$, let $F_i=F[V_1\cup \dots \cup V_i]$, noting that $F_{i+1}$ is formed from $F_i$ by adding a matching of leaves of size $p_in$ for $i=4, \dots, 3m+3$.
Embed $V_i$ to $R_i\cup Q_i$ arbitrarily for $i=1,2,3$. Since $F_3$ has no edges, this gives us an embedding $f_3$ of $F_3$ into $\bigcup_{i=1}^3R_i\cup Q_i$.
For $i=4,\dots, 3m+3$, use Lemma~\ref{Lemma_random_matching_into_independent_sets} to extend $f_{i-1}$ into an embedding $f_i$ of $F_i$ into $\bigcup_{j=1}^{i}R_j\cup Q_j\cup C_j\cup D_j$ (Lemma~\ref{Lemma_2_random_1_deterministic_graph} gives a $C_i\cup D_i$-rainbow matching from $f_{i-1}(V(M_i)\setminus V_i)$ to $R_i\cup Q_i$, which is where we map $M_i$ to get $f_i$). Now $f_{3m+3}$ is a rainbow embedding of $F$.

List the elements of $I=\{v_1, \dots, v_{|I|}\}$ and set $$T_i:=T[V(F)\cup \{v_1, \dots, v_k\}].$$ Note that $T_0=F$ and that $T_{i}$ is formed from $T_{i-1}$ by adding a star centered at $v_i$ whose leaves are in $T_{i -1}$. Set $g_0=f_{3m+3}$. For $i=1, \dots, |I|$, apply Lemma~\ref{Lemma_embed_neighbourhood_random} in order to get a rainbow embedding $g_{i}$ of $T_i$, whose new vertices/colours are in $R_V\cup R_C$ (for this, set $N_i=g_{i-1}(N_{T_i}(v_i))$, $U=V(g_{i-1}(T_{i-1}))\cup C(g_{i-1}(T_{i-1}))$ and use Lemma~\ref{Lemma_embed_neighbourhood_random} to get an $R_C\setminus U$-common neighbour $y_i$ of $N_i$  in $R_V\setminus U$. Then set $g_{i}(v_i):=y_i$. For the application of Lemma~\ref{Lemma_embed_neighbourhood_random}, we use that $n^{\alpha/2}\geq (\Delta +1)n^{1-\alpha} \geq  (\Delta+1) |I|\geq |U\cap (R_V\cup R_C)|$). Now $f:=g_{|I|}$ will be the rainbow embedding of $T$ satisfying the lemma.

For $n^{-\varepsilon}$-uniformity of $f$ on the required sets: note that the embedding is $5n^{-\gamma}$-uniform on $$V_1, V_2, \dots, V_{3m+3}, M_4, M_5, \dots, M_{3m+3}$$ as witnessed by the random sets $$R_1, \dots, R_{3m+3}, C_4, \dots, C_{3m+3}$$ (we have that $f(V_i)\subseteq R_i\cup Q_i$, $|R_i\cup Q_i|\leq p_in+ 2n^{1-\gamma}=|f(V_i)|+ 2n^{1-\gamma}$, and $|Q_i|\leq 2n^{1-\gamma}$ which implies $|f(V_i)\Delta R_i|\leq 5n^{1-\gamma}$. The same argument with $V_i, R_i, Q_i$ replaced by $M_i, C_i, D_i$ shows $|C(f(M_i))\Delta C_i|\leq 5n^{1-\gamma}$). 
Also $f$ is $3n^{-\beta}$-uniform on $I_1, I_2, I_3$ as witnessed by $R_{I_1}, R_{I_2}, R_{I_3}$ (since $|f(I_i)\Delta R_{I_i}|\leq |I_i|+|R_{I_i}|\leq 3n^{1-\beta}$).
Since $V(T),U_1, U_2,  V(T)\setminus (U_1\cup U_2), E(T)$ are each disjoint unions of the sets in  $$\{I_1, I_2, I_3, V_1, V_2, \dots, V_{3m+3}, M_4, M_5, \dots, M_{3m+3}\},$$ we have that $f$ is $(3m\cdot 5n^{-\gamma}+3n^{-\beta})$-uniform on $\{U_1, U_2,  V(T), E(T)\}$. Since  $3m\cdot 5n^{-\gamma}+3n^{-\beta}=15n^{\alpha-\gamma}+3n^{-\beta}\leq n^{-\varepsilon}$, we have that $f$ is  $n^{-\varepsilon}$-uniform on these sets also.
\end{proof}
\vspace{15mm}
\end{document}